\documentclass[11pt]{amsproc}
\usepackage{amsthm,amssymb,paralist,mathrsfs,graphicx,tikz,bbm,upgreek,float}
\usepackage[colorlinks=true,linkcolor=magenta,citecolor=cyan,urlcolor=olive]{hyperref}
\usepackage[normalem]{ulem}

\newcommand{\R}{\mathbb{R}}
\newcommand{\e}{\mathrm{e}}
\newcommand{\Id}{\mathrm{Id}}
\newcommand{\indicator}{\mathbbm{1}}
\newcommand{\supp}{\mathop{\mathrm{supp}}}
\newcommand{\arcosh}{\mathop{\mathrm{arcosh}}}
\newcommand{\rank}{\mathop{\mathrm{rank}}}
\newcommand{\Hess}{\mathop{\mathrm{Hess}}}
\newcommand{\Z}{\mathbb{Z}}

\newcommand{\fm}{\mathfrak{m}}
\newcommand{\0}{\mathbf{0}}
\newcommand{\CK}{\mathscr{A}}
\newcommand{\CE}{\mathscr{E}}
\newcommand{\CR}{\mathfrak{R}}

\newcommand{\bc}{{ \mathcal{H} }}
\newcommand{\fc}{\mathfrak{w}}
\newcommand{\J}{\mathfrak{J}}
\newcommand{\F}{\mathscr{F}}
\newcommand{\fe}{\mathfrak{E}}
\newcommand{\fa}{\mathfrak{a}}
\newcommand{\CC}{\mathbb{C}}
\renewcommand{\Re}{\mathop{\mathrm{Re}}}

\renewcommand{\SS}{\mathbb{S}}
\renewcommand{\H}{\mathbb{H}}
\renewcommand{\i}{\mathrm{i}}
\renewcommand{\d}{\mathrm{d}}
\numberwithin{equation}{section}
\numberwithin{figure}{section}
\newtheorem{thm}{Theorem}

\newtheorem{lem}[thm]{Lemma}

\newtheorem{prop}[thm]{Proposition}

\theoremstyle{definition}

\theoremstyle{remark}
\newtheorem{rk}[thm]{Remark}
\numberwithin{thm}{section}

\allowdisplaybreaks

\title{The Spherical Maximal Operators  on Hyperbolic Spaces }
\author[P. Chen, M. Shen, Y. Wang, L. Yan]{Peng Chen, Minxing Shen, Yunxiang Wang, Lixin Yan}

\address{Peng Chen, Department of Mathematics, Sun Yat-sen University, Guangzhou, 510275, P. R. China}
\email{\href{mailto: chenpeng3@mail.sysu.edu.cn}{chenpeng3@mail.sysu.edu.cn}}
\address{Mingxing Shen, Chern Institute of Mathematics, Nankai University, Tianjin, 300071, P. R. China}
\email{\href{mailto: 9820240068@nankai.edu.cn}{9820240068@nankai.edu.cn}}
\address{Yunxiang Wang, Department of Mathematics, Sun Yat-sen
University, Guangzhou, 510275, P. R. China}
\email{\href{mailto: wangyx93@mail2.sysu.edu.cn}{wangyx93@mail2.sysu.edu.cn}}
\address{Lixin Yan, Department of Mathematics, Sun Yat-sen University, Guangzhou, 510275, P. R. China}
\email{\href{mailto: mcsylx@mail.sysu.edu.cn}{mcsylx@mail.sysu.edu.cn}}

\date{\today}
\subjclass{42B25, 22E30, 35S30, 43A90}

\keywords{Spherical  maximal     operators,    hyperbolic spaces, Fourier integral operators,  Fourier multipliers, local smoothing.}

\begin{document}

\begin{abstract}
In this article we investigate $L^p$ boundedness of  the spherical maximal operator $\mathfrak{m}^\alpha$ of (complex) order $\alpha$ on the $n$-dimensional hyperbolic space $\mathbb{H}^n$, which was introduced and studied by El Kohen. We prove that when $n\geq 2$, for $\alpha\in\mathbb{R}$ and $1<p<\infty$, if $\mathfrak{m}^\alpha$ is bounded on $L^p(\mathbb{H}^n)$, then we must have $\alpha>1-n+n/p$ for $1<p\leq 2$; or $\alpha\geq \max\{1/p-(n-1)/2,(1-n)/p\}$ for $2<p<\infty$. Furthermore, we improve El Kohen's result [J. Operator Theory 3 (1980)] on $L^p$ boundedness of $\mathfrak{m}^\alpha$ by showing that $\mathfrak{m}^\alpha$ is bounded on $L^p(\mathbb{H}^n)$ provided that
$\mathop{\mathrm{Re}}\alpha> \max \{{(2-n)/p}-{1/(p p_n)},{(2-n)/p}- (p-2)/[p p_n(p_n-2)]\}
$ for $2\leq p\leq \infty$, with $p_n=2(n+1)/(n-1)$ for $n\geq 3$ and $p_n=4$ for $n=2$.
\end{abstract}
\maketitle

%%%%%%%%%%%%%%%%%%%%%%%%%%%%%%%%%%%%%%%%%%%%%%%%%%%%%%%%%%%%%%%%%%%%%%%%%%%%%%%%%%

\section{Introduction}\label{sec:1}
\setcounter{equation}{0}

\subsection{Background}
The spherical maximal operator of (complex) order $\alpha$ on ${\mathbb R^n}$ was introduced in 1976 by Stein \cite{St2},  which is defined by
\begin{eqnarray}\label{e1.1}
{\frak M}^\alpha f(x)= \sup_{t>0}
|{\frak M}^\alpha_tf (x) |,
\end{eqnarray}
where ${\frak M}^\alpha_t$ is the spherical operator given by
\begin{eqnarray}\label{e1.2}
{\frak M}^\alpha_t f (x) =
{1\over   \Gamma(\alpha)  }  \int_{\{|y|\leq 1\}} \left(1-{|y|^2 }\right)^{\alpha -1} f(x-ty)\,\d y.
\end{eqnarray}
In the  case $\alpha=1$,  ${\frak M}^\alpha$ corresponds to the Hardy-Littlewood maximal operator and in the case  $\alpha=0$, one  recovers the spherical maximal operator  ${\frak M}  f(x)= \sup_{t>0}
|{\frak M}_tf (x) |$ in which
\begin{eqnarray}\label{e1.3}
{\frak M}_tf (x) = {1\over\omega_{n-1}} \int_{{\mathbb S}^{n-1}}   f(x-t\omega) \,\text{d}\omega, \ \ \ x\in   \R^n,\ t>0,
\end{eqnarray}
where ${\mathbb S}^{n-1}$ denotes the standard unit sphere in ${\mathbb R^n}$, $\omega_{n-1}$ is the surface area of $\SS^{n-1}$ and $\text{d}\omega$ is the induced Lebesgue measure on the unit sphere ${\mathbb S}^{n-1}.$ The following brief summary gives an overview of the progress
so far in this direction of research.

\smallskip
(i)\, In \cite[Theorem 2]{St2}, Stein showed that for $n\geq 2$
\begin{eqnarray}\label{e1.4}
\|	{\frak M}^{\alpha}  f \|_{L^p({\mathbb R^n})} \leq C\|f \|_{L^p({\mathbb R^n})}
\end{eqnarray}
provided that
$
\alpha>1-n+n/p$  when $1<p\leq 2;
$
or
$
\alpha>(2-n)/p$  when $2\leq p\leq \infty.
$
The above maximal theorem tells us  that    when  $\alpha=0$ and $n\geq 3$,  the    operator ${\frak  M} $ is bounded on $L^p({\mathbb R^n})$ for $p>n/(n-1)$.
This range of $p$ is sharp. As is pointed out in \cite{St2, St1}, no such result can hold for $p\leq  n/(n-1)$ if $n\geq 2$.

\smallskip

(ii)\, In \cite{Bo},  Bourgain   proved that  when  $\alpha=0$ and $n=2,$ the operator    ${\frak M} $ is bounded on $L^p{(\mathbb R^2)} $ for $p>2$.    Bourgain's theorem indicates that there exists $\varepsilon(p)>0$ such that
\begin{eqnarray}\label{e1.5}
\|	{\frak M}^{\alpha}  f \|_{L^p({\mathbb R^2})} \leq C\|f \|_{L^p({\mathbb R^2})}, \ \ \  \Re \alpha > -\varepsilon(p), \ \ \  2< p< \infty.
\end{eqnarray}
This result cannot hold even for $\alpha=0$ when $p=2$, see \cite{St2,St1}.
Subsequently Mockenhaupt, Seeger and Sogge \cite{MSS} gave an alternative proof of Bourgain's result by using local smoothing estimates for the wave operator in $2+1$ dimensions.

\smallskip

(iii)\, Using  the Bourgain-Demeter decoupling theorem \cite{BD}, Miao, Yang and Zheng \cite{MiYaZh} extended certain range of $\alpha$ in Stein's result \cite[Theorem 2]{St2} by showing that for $  n\geq 2$ and $p>2$,  \eqref{e1.4} holds
whenever
$
\Re\alpha >\max\left\{{ (1-n)/4} +{(3-n)/(2p)}, \,  {(1-n)/p} \right\}.
$

\smallskip

(iv)\, Recently,   Liu,  Song,   and the second and  fourth named authors  \cite{LiShSoYa} obtained the necessary condition on $(\alpha,p)$ for the   operator $	{\frak M}^\alpha$ to be bounded on $L^p(\R^n)$ by showing that for $n\geq 2$ and $p>2$ \eqref{e1.4}
holds only if
$
\Re\alpha \geq   \max \{1/p-(n-1)/2,\ -(n-1)/p \}.
$
In the two dimensional case   $n=2$, they showed  that  \eqref{e1.4} holds    whenever   $\Re\alpha>\max\{1/p-1/2,\ -1/p\}$  by applying    the work  of Guth-Wang-Zhang \cite{GWZ} on  local smoothing estimates for the wave operator  in $2+1$ dimensions.

Note that in \cite{NRS},   Nowak, Roncal and Szarek  \cite{NRS} found 
 sharp conditions 
for the    operator ${\frak M}^\alpha$ on radial functions to be bounded on 
 $L^p_{\rm rad}({\mathbb R}^n)$. They
  proved that for $n\geq 2 $ and $ \alpha> (1-n)/2$,    $	{\frak M}^\alpha$  is bounded on $L^p_{\rm rad}({\mathbb R}^n)$ if and only if  
$\alpha> 1-n + n/p$ if $1<p\leq 2$; $\alpha> 1/p + (1-n)/2$ if $2\leq p \leq 2n/(n-1)$; and $\alpha\geq (1-n)/p$ if $p> 2n/(n-1)$.   

\smallskip

In the last decades   the spherical maximal operators
have attracted a lot of attention and have been studied extensively by many authors.
The maximal theorem for spherical operator has some applications; for example, Stein \cite{St2} used it to derive a Fatou's theorem for wave equation.
The spherical operator can be extended to a more general hypersurface in   ${\mathbb R^n}$, see \cite{Bo, G, Io, IS,  MSS, So2,So1,   SS1, SoSt, St1,   SWa} for the references therein.

\smallskip

\subsection{Main results}

The purpose of this article is to investigate
the spherical maximal operators of (complex) order $\alpha$ on the $n$-dimensional  hyperbolic space $\H^n$, which was  introduced and studied by El Kohen (\cite{Ko}).  To do this,  we denote the Minkowski metric on ${\mathbb R^{n+1}}$   by
$$
[z,w]=z_0w_0 -z_1w_1-\cdots -z_nw_n,\ \ \  \  z =(z_0,\dots, z_n), w =(w_0,\dots,w_n)\in {\mathbb R^{n+1}}.
$$
The  hyperbolic space $\H^n$ is given by
\begin{eqnarray*}
\H^n=\{z=(z_0,z')\in {\mathbb R^{n+1}}: [z]:=[z, z]=1, \  z_0>0\},
\end{eqnarray*}
and we equip this space  with the  Riemannian metric induced by the above Minkowski metric  $[\cdot,\cdot]$.  This Riemannian metric induces in turn a measure, by which we denote $\d z$. Define the  spherical maximal operators by
\begin{eqnarray}\label{e1.7}
{\mathfrak m}^\alpha(f)(z)  =\sup_{t>0}|M^\alpha_t(f)(z)|,
\end{eqnarray}
where
\begin{eqnarray}\label{e1.8}
M^\alpha_t(f)(z) =2\e^t\left({\e^t-1\over \sinh t}\right)^{n-2}{1\over(\e^t-1)^{2\alpha+n-2}} {1\over \Gamma(\alpha)}\int_{B_t(z)}[\e^tz-w]^{\alpha-1}f(w)\,\d w,
\end{eqnarray}
$B_t(z)$ stands for the geodesic ball in $\H^n$ centered at $z\in\H^n$ with radius $t$ and $f$ is in the Schwartz space on $\H^n$.  These operators $\fm^\alpha$ and $M^\alpha_t$ are initially defined for $\Re\alpha >0,$   but the definition can be extended to all complex $\alpha$ with   $\Re\alpha> (1-n)/2$ by analytic continuation, see Lemma~\ref{lem:2.2} below. For $\alpha=0,$ one recovers the spherical maximal operator
\begin{eqnarray}\label{e1.9}
{\mathfrak m} (f)(z)
&=& \sup_{t>0} |f\ast \d\sigma_t(z)|,
\end{eqnarray}
where $\d\sigma_t$ is the normalized spherical measure defined on \cite[p. 277]{Io}.

A natural question is for which $\alpha$ the spherical maximal opeators ${\mathfrak m}^\alpha$  are bounded on $L^p$. In \cite[Theorem 3]{Ko}, El Kohen showed that for $n\geq 2$,
\begin{eqnarray}\label{e1.10}
\|{\mathfrak m}^\alpha f\|_{L^p(\H^n)}\leq C\| f\|_{L^p(\H^n)}
\end{eqnarray}
holds under the following circumstances:
\begin{eqnarray}\label{e1.11}
{\rm Re}\, \alpha>1-n+ {n\over p} \ \ \ \     {\rm when}\     1<p\leq 2;
\end{eqnarray}
or
\begin{eqnarray}\label{e1.12}
\ \ {\rm Re}\, \alpha>{2-n\over p}   \ \ \ \   \ \ \ \  {\rm when}\   2< p\leq \infty.
\end{eqnarray}
It was not until 2000 that Ionescu \cite{Io} showed that ${\mathfrak m}^0 $ is bounded on $L^p(\H^2)$ for $2<p\leq \infty$ by making use of the local smoothing estimates for Fourier integral operators (\cite{So2}).
The above two admissible relations for $\alpha$ and $p$ are summarized when $n\geq 3$ in Figure \ref{fig:1}
below, where the relations \eqref{e1.11} and \eqref{e1.12} correspond to the dashed segments $AB$ and $OB$, respectively.

\begin{figure}[H]
\begin{tikzpicture}[x=1.00mm, y=1.00mm, inner xsep=0pt, inner ysep=0pt, outer xsep=0pt, outer ysep=0pt]
\path[line width=0mm] (32.13,17.11) rectangle +(70.36,83.05);
\definecolor{L}{rgb}{0,0,0}
\path[line width=0.15mm, draw=L] (49.95,49.95) -- (100.05,49.95);
\definecolor{F}{rgb}{0,0,0}
\path[line width=0.15mm, draw=L, fill=F] (100.05,49.95) -- (98.65,50.30) -- (100.05,49.95) -- (98.65,49.60) -- (100.05,49.95) -- cycle;
\path[line width=0.15mm, draw=L] (50.00,20.00) -- (50.00,94.97);
\path[line width=0.15mm, draw=L, fill=F] (50.00,94.97) -- (49.65,93.57) -- (50.00,94.97) -- (50.35,93.57) -- (50.00,94.97) -- cycle;
\draw(100.49,45.06) node[anchor=base east]{\fontsize{11.38}{13.66}\selectfont $1\over p$};
\draw(46.38,95.03) node[anchor=base west]{\fontsize{11.38}{13.66}\selectfont $\alpha$};
\path[line width=0.15mm, draw=L] (49.93,49.92) circle (0.73mm);
\path[line width=0.15mm, draw=L] (70.00,30.00) circle (0.68mm);
\path[line width=0.15mm, draw=L] (89.96,90.01) circle (0.70mm);
\draw(48.88,88.60) node[anchor=base east]{\fontsize{11.38}{13.66}\selectfont 1};
\path[line width=0.15mm, draw=L, dash pattern=on 2.00mm off 1.00mm] (49.94,90.00) -- (90.00,90.00);
\draw(89.99,46.09) node[anchor=base]{\fontsize{11.38}{13.66}\selectfont 1};
\path[line width=0.15mm, draw=L, dash pattern=on 2.00mm off 1.00mm] (90.00,90.01) -- (90.00,49.95);
\path[line width=0.60mm, draw=L, dash pattern=on 2.00mm off 1.00mm] (89.98,90.07) -- (70.02,29.96);
\path[line width=0.60mm, draw=L, dash pattern=on 2.00mm off 1.00mm] (70.01,30.05) -- (49.99,49.99);
\path[line width=0.15mm, draw=L, dash pattern=on 2.00mm off 1.00mm] (70.01,30.01) -- (70.01,49.89);
\draw(70.00,52.55) node[anchor=base]{\fontsize{11.38}{13.66}\selectfont $1\over 2$};
\path[line width=0.15mm, draw=L, dash pattern=on 2.00mm off 1.00mm] (70.06,30.01) -- (50.01,30.01);
\draw(43.18,29.11) node[anchor=base west]{\fontsize{11.38}{13.66}\selectfont $2-n\over 2$};
\path[line width=0.15mm, draw=L] (59.98,36.34) circle (0.73mm);
\path[line width=0.15mm, draw=L, dash pattern=on 2.00mm off 1.00mm] (59.99,36.20) -- (60.01,49.84);
\draw(59.31,52.55) node[anchor=base]{\fontsize{11.38}{13.66}\selectfont $1\over p_n$};
\path[line width=0.15mm, draw=L, dash pattern=on 2.00mm off 1.00mm] (59.98,36.34) -- (49.96,36.35);
\draw(34.13,35.28) node[anchor=base west]{\fontsize{11.38}{13.66}\selectfont ${2-n\over p_n}-{1\over p_n^2}$};
\path[line width=0.60mm, draw=L, dash pattern=on 0.60mm off 0.50mm] (49.91,50.02) -- (60.06,36.33);
\path[line width=0.60mm, draw=L, dash pattern=on 0.60mm off 0.50mm] (59.99,36.30) -- (70.02,30.01);
\path[line width=0.15mm, draw=L, dash pattern=on 2.00mm off 1.00mm] (63.81,23.71) -- (50.01,23.66);
\path[line width=0.15mm, draw=L, fill=F] (63.72,23.75) circle (0.73mm);
\draw(36.30,22.83) node[anchor=base west]{\fontsize{11.38}{13.66}\selectfont $-{(n-1)^2\over 2n}$};
\path[line width=0.60mm, draw=L] (49.96,49.93) -- (63.73,23.74);
\path[line width=0.60mm, draw=L] (63.72,23.75) -- (70.01,30.02);
\draw(49.01,50.65) node[anchor=base east]{\fontsize{11.38}{13.66}\selectfont $O$};
\draw(64.59,52.55) node[anchor=base]{\fontsize{11.38}{13.66}\selectfont $n-1\over 2n$};
\path[line width=0.15mm, draw=L, dash pattern=on 2.00mm off 1.00mm] (63.73,49.94) -- (63.73,23.78);
\draw(91.38,90.56) node[anchor=base west]{\fontsize{11.38}{13.66}\selectfont $A$};
\draw(71.03,27.13) node[anchor=base west]{\fontsize{11.38}{13.66}\selectfont $B$};
\draw(61.24,31.53) node[anchor=base]{\fontsize{11.38}{13.66}\selectfont $C$};
\draw(64.48,19.98) node[anchor=base west]{\fontsize{11.38}{13.66}\selectfont $D$};
\end{tikzpicture}
\caption{The $(1/p,\alpha)$-plot for $n\geq 3$. El Kohen's result \cite[Theorem 3]{Ko} indicates that $\fm^\alpha$ is bounded on $L^p(\H^n)$ if $(1/p,\alpha)$ is strictly above the dashed folded  segments $OBA$. We extend the range of $(1/p,\alpha)$ in El Kohen's result for $2<p<\infty$ to the range strictly above the dotted folded segments $OCB$. In addition, we showed that $\fm^\alpha$ is bounded on $L^p(\H^n)$ only if $(1/p,\alpha)$ is on or above the solid folded segments $ODB$ for $2<p<\infty$, or strictly above the dashed segment $BA$ for $1<p\leq 2$.}
\label{fig:1}
\end{figure}
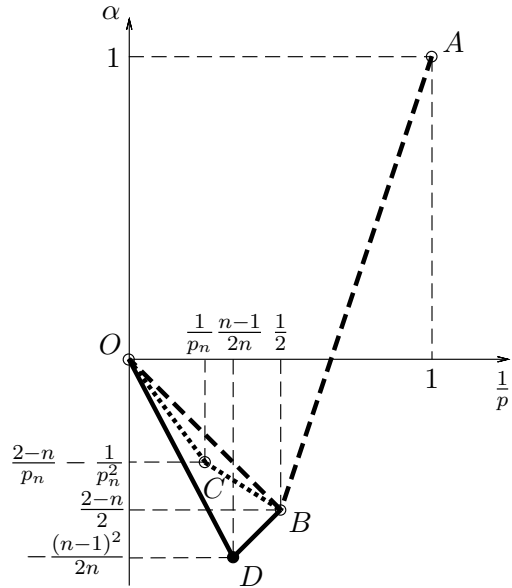

The present article  can be viewed as a continuation of the above body of work in   \cite{Io, Ko}.
Firstly,  we give a necessary condition on $(\alpha,p)$ for $\fm^\alpha$ to be bounded on $L^p(\H^n)$
for $\alpha\in\R$.

\begin{thm}\label{thm:1.1}
Let $\alpha\in\R$ and $n\geq 2$. Suppose  $\fm^\alpha$ is bounded on $L^p(\H^n)$ for $1<p<\infty$. Then we must have:
\begin{eqnarray}\label{e1.13}
\alpha>1-n+ {n\over p} \ \ \ \     {\rm when}\     1<p\leq 2;
\end{eqnarray}
or
\begin{eqnarray}\label{e1.14}
\alpha\geq  \max\left\{{1\over p}-{n-1\over 2},\ -{n-1\over p}\right\}\ \ \ \     {\rm when}\     2<p\leq \infty.
\end{eqnarray}
\end{thm}

Secondly, we extend the range of $\alpha$ in \eqref{e1.12} for $2<p< \infty$  to the following result.

\begin{thm}\label{thm:1.2}
Suppose $2< p< \infty$ and $n\geq 2$. Then $\fm^\alpha$ is bounded on $L^p(\H^n)$ provided that
\begin{eqnarray*}
\Re\alpha>\max\left\{{2-n\over p}-{1\over p p_n}, \ {2-n\over p} - {p-2\over p p_n(p_n-2) }\right\},
\end{eqnarray*}
where
\begin{eqnarray}\label{e1.15}
p_n=\begin{cases}
\displaystyle{2(n+1)\over n-1}&\mbox{for }n\geq 3,\\
4&\mbox{for }n=2.
\end{cases}
\end{eqnarray}
\end{thm}

Theorems  ~\ref{thm:1.1}  and  ~\ref{thm:1.2} can be illustrated using Figure \ref{fig:1} above. The relations \eqref{e1.13} and \eqref{e1.14} in Theorem \ref{thm:1.1} correspond to the dashed segment $AB$ and the solid folded segments $ODB$, respectively. Theorem \ref{thm:1.2} indicates that compared to \eqref{e1.12}, there is an $1/(pp_n)$-downwards extension of the range of $\Re\alpha$ so that \eqref{e1.10} is valid for $p\geq p_n$, and there is a $(p-2)/[pp_n(p_n-2)]$-improvement for $2<p<p_n$. This is indicated in Figure \ref{fig:1} above by the dotted folded segments $OCB$.

The paper is organized as follows. In Section \ref{sec:2} we give some preliminary  facts about the space $\H^n$ as well as its Fourier theory. In Section \ref{sec:4} we show Theorem \ref{thm:1.1} by constructing three examples. In  Section \ref{sec:5} we  prove Theorem \ref{thm:1.2} by applying an approach of Ionescu \cite{Io}, whose core is to adopt the Iwasawa coordinates to use the work of Beltran-Hickman-Sogge \cite{BeHiSo} and Gao-Liu-Miao-Xi \cite{GaLiMiXi} on local smoothing estimates for Fourier integral operators on ${\mathbb R^n}$.

\medskip

%%%%%%%%%%%%%%%%%%%%%%%%%%%%%%%%%%%%%%%%%%%%%%%%%%%%%%%%%%%%%%%%%%%%%%%%%%%%%%%%%%

\section{Fourier analysis on the hyperbolic space}\label{sec:2}
\setcounter{equation}{0}

In  this section  we recall some basic facts about the hyperbolic space $\H^n$ as well as its Fourier theory (see for examples, \cite{He, IoSt, Pe}).
The Minkowski metric on ${\mathbb R^{n+1}}$ is given by
$$
[z,w]=z_0w_0 -z_1w_1-\cdots -z_nw_n,\ \ \   \ \  z=(z_0,\dots, z_n), \ w= (w_0,\dots,w_n)\in {\mathbb R^{n+1}}.
$$
We define $\H^n$ as the hyperboloid
\begin{eqnarray*}
\H^n=\{z=(z_0,z')\in {\mathbb R^{n+1}}: [z, z]=1, \  z_0>0\}
\end{eqnarray*}
and equip this space with the Riemannian metric induced by the Minkowski metric $[\cdot,\cdot]$. This Riemannian metric induces in turn a measure, by which we denote $\d z$. We will further distinguish the point ${\bf 0}=(1, 0, \dots, 0)$.

The Laplace-Beltrami operator on $	\H^n$ is denoted by $\Delta_{	\H^n}$. It has spectrum $(-\infty, -(n-1)^2/4]$. For $\omega\in {\mathbb S}^{n-1}$, let
$$
b(\omega)=(1, \omega)\in {\mathbb R}^{n+1}.
$$
The analog of plane waves is provided by
$$
h_{\lambda, \omega}(z)=[z, b(\omega)]^{i\lambda-(n-1)/2}, \ \ \ z\in 	\H^n
$$
(notice that $[z, b(\omega)]>0 $ for $z\in 	\H^n$). They satisfy
$$
\Delta_{	\H^n}h_{\lambda, \omega}(z)=-\left(\lambda^2 +{(n-1)^2\over 4}\right)\, h_{\lambda, \omega}(z).
$$

The Helgason Fourier transform on $\H^n$ is defined as
\begin{equation}\label{e2.111}
{\F}(f)(\lambda, \omega)=\int_{\H^n} f(z)\, h_{\lambda, \omega}(z)\,\d z, \ \ \ \  (\lambda, \omega) \in {\mathbb R}_+ \times {\mathbb S}^{n-1},
\end{equation}
while the inverse Fourier transform is then given by
\begin{equation}\label{e2.112}
f(z)= {1\over \omega_{n-1}} \int_0^{\infty}\int_{\mathbb S^{n-1}}  {\F}(f)(\lambda, \omega)\, {\overline{
h_{\lambda, \omega}(z)}}\, |\bc(\lambda)|^{-2}\,\d\lambda\d\omega,
\end{equation}
for the Harish-Chandra function
\begin{eqnarray}\label{e2.1}
\bc(\lambda)= {2^{n-2}\Gamma(n/2)\over \sqrt{\uppi}}{\Gamma(\i\lambda)\over\Gamma((n-1)/2+\i\lambda)}.
\end{eqnarray}
This time the following Plancherel's theorem holds: the Fourier transform is an isometry from $L^2(\H^n, \d z)$ to $L^2({\mathbb R_+}\times {\mathbb S^{n-1}},  |\bc(\lambda)|^{-2} \d\lambda {\d\omega \over \omega_{n-1}} )$.

We need the estimates for the Harish-Chandra function $\bc$, which will be frequently used throughout the paper.
For its proof, we refer the reader to \cite[Proposition A1(a)(b)]{Io}.

\begin{lem}\label{lem:2.0}
Suppose $\lambda\in \R$.
The Harish-Chandra function $\bc(\lambda)$ satisfies
$
|\bc(\lambda)|^{-2}=(\bc(\lambda)\,\bc(-\lambda))^{-1}.
$
The function $ \lambda^{-1}(\bc(-\lambda))^{-1}$ belongs to $C^\infty(\R)$ and
\begin{eqnarray*}
\left|{\d^k\over\d\lambda^k}(\lambda^{-1}(\bc(\lambda))^{-1})\right|\leq C_k\,(1+|\lambda|)^{(n-1)/2-1-k}.
\end{eqnarray*}
\end{lem}

In analogy to $\bc$, we define another function
\begin{eqnarray}\label{e2.2}
\bc^\alpha(\lambda)={2^{n-2}\Gamma(n/2)\over \sqrt{\uppi}}{\Gamma(\i\lambda)\over\Gamma((n-1)/2+\alpha+\i\lambda)}, \ \ \  {\rm Re}\, \alpha>{(1-n)/2}.
\end{eqnarray} 
Specifically, when $\alpha=0$, $\bc^\alpha(\lambda) $ is exactly the Harish-Chandra function $\bc(\lambda)$ defined in \eqref{e2.1}. For the function $\bc^\alpha$ we have the following result.

\begin{lem}\label{lem:2.1}
Let  ${\rm Re}\, \alpha>{(1-n)/2}$ and $k\in{\mathbb N}$.  Then the  function $  \lambda\,\bc^\alpha(\lambda)$ belongs to $C^\infty(\R)$ and
\begin{eqnarray*}
\left|{\d^k\over\d\lambda^k}(\lambda \,\bc^\alpha(\lambda))\right|\leq C_{\alpha,k}\,(1+|\lambda|)^{1-(n-1)/2-\Re\alpha-k}.
\end{eqnarray*}

\end{lem}

\begin{proof}
This is a consequence of  Stirling's formula (\cite[p. 151, Example. (ii)]{Ti}).
\end{proof}
\smallskip

In the space $\H^n$, we can adopt polar coordinates
\begin{eqnarray*}
z=(\cosh r,\omega\sinh r),  \ \ \  \ \ \   (r, \omega) \in {\mathbb R}_+ \times {\mathbb S}^{n-1},
\end{eqnarray*}
where $\SS^{n-1}\subset \R^n$ is the unit sphere on $\R^n$. The canonical Riemannian metric on $\H^n$ is then given by
\begin{eqnarray*}
\d r^2+\sinh^2 r\,\d s_{n-1}^2,
\end{eqnarray*}
where $r\geq 0$ is the geodesic distance to the orginal ${\bf 0}$ and $\d s^2_{n-1}$ is the canonical Riemannian metric on $\SS^{n-1}$. In these coordinates, the volume element becomes
$$\d z= (\sinh r)^{n-1}\, \d r \d\omega.
$$

The spherical function on $\H^n$ is defined by
\begin{eqnarray*}
\varphi_\lambda(z)={1\over \omega_{n-1}} \int_{{\mathbb S^{n-1}}} h_{\lambda, \omega}(z)\, \d\omega.
\end{eqnarray*}
It depends only on the distance $r$ of $z$ to the origin $\0$ and can be written as
\begin{eqnarray}\label{e2.4}
\varphi_\lambda(r)={\Gamma(n/2)\over \sqrt{\uppi}\,\Gamma((n-1)/2)\,}    \int_0^\uppi (\cosh r-\cos s\,\cosh r)^{-(n-1)/2+\i\lambda}(\sin s)^{n-2}\,\d s.
\end{eqnarray}
See \cite{Vi}.

Since radial   functions on $\H^n$ depend on $r$ only, the Helgason Fourier transform \eqref{e2.111} of such function $f$ can be written as
$$
{\F}(f)(\lambda)=\int_{\H^n}f(z)\,\varphi_\lambda(z)\,\d z=\omega_{n-1}\int_{\H^n} f(r)\, \varphi_\lambda(r) \,(\sinh r)^{n-1}\, \d r
$$
and the inverse Fourier transform \eqref{e2.112} is  given by
$$
f(r)=  \int_0^{\infty} {\F}(f)(\lambda )\,
\varphi_\lambda(r)\, |\bc(\lambda)|^{-2}\, \d\lambda.
$$
If $K$ is radial, then the convolution of $f$ and $K$ on $\H^n$ can be written as
$$
f\ast K(z)=\int_{\H^n} f(w) \,K(d(z, w)) \,\d w,
$$
where $d(z,w)= \arcosh ([z, w])$  denotes the geodesic distance between $z$ and $w$.

In the following, we denote  $D=\sqrt{-\Delta_{	\H^n} -(n-1)^2/4}$.
For any bounded Borel function $m:[0, \infty)\to {\mathbb C}$,  we define the
radial Fourier multipliers $m(D)$ by the formula
\begin{eqnarray}\label{e2.3}
m(D)(f)(z)= {\F}^{-1}[ m(\lambda)\, {\F}(f)(\lambda, \omega)](z).
\end{eqnarray}
See \cite[p. 10]{GeLe}. It follows from  \eqref{e2.3} that  the
radial Fourier multipliers $m(D)$  can be rewritten in the following:
\begin{eqnarray*}
m(D)(f)(z)=f*K(z)=\int_{\H^n}f(w)\,K(d(z,w))\,\d w,
\end{eqnarray*}
where
\begin{eqnarray*}
K(r)=\F^{-1}(m)(r)=\int_0^\infty m(\lambda)\,\varphi_\lambda(r)\,|\bc(\lambda)|^{-2}\,\d\lambda.
\end{eqnarray*}
\smallskip

As mentioned in Section \ref{sec:1}, the  operator $M^\alpha_t(f)$ in  \eqref{e1.8}  is  initially defined for $\Re \alpha >0,$   but the definition can be extended to all complex $\alpha$ with   $\Re \alpha> (1-n)/2$ by analytic continuation.
To explain this, we shall use the Legendre function. Recall that
in view of \cite[p. 155, (1)]{Er}, for $\Re(-\mu)\geq \Re\nu>-1$ and $\zeta\notin[-1,\infty)$   the Legendre function of order $\mu$ and degree $\nu$ has the following integral representation:
\begin{eqnarray*}
P^\mu_\nu(\zeta)={2^{-\nu}(\zeta^2-1)^{-\mu/2}\over \Gamma(-\mu-\nu)\,\Gamma(\nu+1)}\int_0^\infty (\zeta+\cosh s)^{\mu-\nu-1}(\sinh s)^{2\nu+1}\,\d s.
\end{eqnarray*}
Then we have the following result.

\begin{lem}\label{lem:2.2}
The operator   $M^\alpha_t(f)$   in  \eqref{e1.8} can be extended to
$\alpha\in\CC$ with $\Re\alpha>(1-n)/2$  via analytic continuation such that
\begin{eqnarray*}
M^\alpha_t(f)(z)=m^\alpha_t(D)(f)(z),
\end{eqnarray*}
where
\begin{eqnarray}\label{e2.6}
m^\alpha_t(\lambda)=2^{(n-2)/2+\alpha}  \Gamma\! \left({n\over 2}\right)
\e^{\alpha t}  (\e^t -1)^{-2\alpha} (\sinh t)^{\alpha-(n-2)/2}P^{-\alpha -(n-2)/2}_{-1/2+\i\lambda}(\cosh t).
\end{eqnarray}
\end{lem}

\begin{proof}
In  \cite[Theorem 1]{Ko}, El Kohen showed that for $\alpha\in\CC$ with $\Re\alpha>(1-n)/2$ one can write
\begin{eqnarray*}
M^\alpha_t(f)(z)=m^\alpha_t(D)(f)(z)
\end{eqnarray*}
via analytic continuation, where
\begin{equation}\label{e2.5}
m^\alpha_t(\lambda)=2^{{n-2\over 2}+\alpha}  \Gamma\! \left({n\over 2}\right)
\e^{\alpha t}  (\e^t -1)^{-2\alpha} (\sinh t)^{-(n-2)}I^{\alpha +{n-2\over 2}}P^0_{-{1\over 2}+\i\lambda}(\cosh t)
\end{equation}
and $I^\gamma$ is the fractional integral operator given by
\begin{eqnarray*}
I^\gamma(f)(s)={1\over\Gamma(\gamma)}\int_1^sf(t)\,(s-t)^{\gamma-1}\,\d t.
\end{eqnarray*}
In view of \cite[(5) on p. 159]{Er}, we have
\begin{eqnarray*}
I^{\alpha+(n-2)/2}P^0_{-1/2+\i\lambda}(\cosh t)=(\sinh t)^{\alpha+(n-2)/2}P^{-\alpha-(n-2)/2}_{-1/2+\i\lambda}(\cosh t).
\end{eqnarray*}
Substituting this back into \eqref{e2.5},   we obtain \eqref{e2.6}.
\end{proof}

%In particular, when $\alpha=0$, $m^\alpha_t(\lambda)$ is exactly a constant multiple of the spherical function $\varphi_\lambda(t)$ defined in \eqref{e2.4}. For general $\alpha$, w
We have the following asymptotics of the multiplier function $m^\alpha_t(\lambda)$, which will be used in the proof of Theorem~\ref{thm:1.2}.

\begin{lem}\label{lem:2.3}
Suppose $\lambda\in\R$.
\begin{asparaenum}[(i)]
\item If $t>0$, then
\begin{eqnarray*}
|m^\alpha_t(\lambda)|\leq C_\alpha \,(1+t)\,\e^{-(n-1) t/2}.
\end{eqnarray*}

\item If $0<t\leq \uppi$ satisfying $t|\lambda|\geq 1$, then $m^\alpha_\lambda(t)$ can be written in the form
\begin{eqnarray}\label{emmm}
m^\alpha_t(\lambda)=\e^{\i\lambda t}a^{\alpha,N}_1(\lambda,t)+\e^{-\i\lambda t}a^{\alpha,N}_1(-\lambda,t)+E^{\alpha,N}(\lambda,t),
\end{eqnarray}
where
\begin{eqnarray}\label{e2.601}
\begin{cases}
\displaystyle|{\partial^k_\lambda}{\partial^l_t}a^{\alpha,N}_1(\lambda,t)|\leq C_N\,(|\lambda|t)^{-\Re \alpha-(n-1)/2}|\lambda|^{-k}t^{-l},\\[4pt]
|E^{\alpha,N}(\lambda,t)|\leq C_N\,(|\lambda|t)^{-\Re\alpha-(n-1)/2-N-1} 
\end{cases}
\end{eqnarray}
for all integers $N, k, l\geq 0$, and for $\lambda, t$ in the ranges stated above.

\item If $t>(\log 2)/2$, then $m^\alpha_t(\lambda)$ can be written in the form
\begin{eqnarray*}
m^\alpha_t(\lambda)=\e^{-(n-1) t/2}\left(\e^{\i\lambda t}\bc^\alpha(\lambda)\,a^\alpha_2(\lambda,t)+\e^{-\i\lambda t}\bc^\alpha(-\lambda)\,a^\alpha_2(-\lambda,t)\right),
\end{eqnarray*}
where $\bc^\alpha$ is as in \eqref{e2.2} and $a^\alpha_2(\lambda,t)\in S^0_\lambda$, i.e.,
\begin{eqnarray*}
|\partial^k_\lambda\partial^l_t a^\alpha_2(\lambda,t)|\leq C_{k,l}(1+|\lambda|)^{-k}.
\end{eqnarray*}
\end{asparaenum}
\end{lem}
\begin{proof}
According to \cite[p. 156, (8)]{Er},
\begin{eqnarray*}
P_{-1/2+\i\lambda}^{-\alpha-(n-2)/2}(\cosh t)
=C_\alpha\, (\sinh t)^{-\alpha-(n-2)/2}\int_0^t(\cosh t-\cosh s)^{\alpha+(n-3)/2}\cos (\lambda s)\,\d s.
\end{eqnarray*}
A direct calculation shows that for any integer $k\geq 0$
\begin{eqnarray*}
|P_{-1/2+\i\lambda}^{-\alpha-(n-2)/2}(\cosh t)|\leq C_\alpha \begin{cases}
t^{\Re \alpha+(n-2)/2}&\mbox{for } 0<t\leq 1,\\
t\,\e^{-t/2}&\mbox{for }t>1.
\end{cases}
\end{eqnarray*}
Substituting this back into \eqref{e2.6},   we obtain  (i).

For part (ii) we invoke \cite[(2.4.1(5))]{Sc} to write that for $t\leq \uppi$,
\begin{equation} \label{e2.7}
(\sinh t)^{1/2}P_{-1/2+\i\lambda}^{-\alpha-(n-2)/2}(\cosh t)
=t^{\alpha+(n-1)/2}\sum_{j=0}^N
t^{2j}b^\alpha_j(t)\,\J_{\alpha+(n-2)/2+j}(\lambda t)+ {\mathfrak R}^\alpha(\lambda,t),
\end{equation}
where    $|\d^l b^\alpha_j(t)/\d t^l|\leq C_{\alpha,l}$ for $l\geq 0$ uniformly in $j\geq 0$
(\cite[(2.4.1(4))]{Sc}),
\begin{eqnarray*}
\J_m(r)=\int_{-1}^1\e^{\i vr}(1-v^2)^{m-1/2}\,\d v
\end{eqnarray*}
and by \cite[(2.4.2(9))]{Sc},
$$
|{\mathfrak R}^\alpha(\lambda,t)|\leq C_{\alpha}(|\lambda|t)^{-\Re \alpha-(n-1)/2-N-1}.
$$
It's well known that for all $m\in \CC$ with $\Re m>-1/2$ and $r$ large,
\begin{eqnarray*}
\J_m(r)=\e^{\i r}\psi_m(r)+\e^{-\i r}\psi_m(-r)+ {\mathfrak O}_m(r),
\end{eqnarray*}
where $\psi_m(r)$ are smooth functions such that for all integer $k\geq 0$,
\begin{eqnarray}\label{e2.701}
\left|{\d^k\psi_m(r)\over\d r^k}\right|\leq C_m|r|^{-m-1/2-k}
\end{eqnarray}
and $|{\mathfrak O}_m(r)|\leq C_m|r|^{-m-1/2-N-1}$. See \cite[p. 51]{Io}.
%The proof (ii) will be finished if we substitute  \eqref{e2.7}  back into \eqref{e2.6} to  set
Now set
\begin{equation}\label{e2.702}
a^{\alpha,N}_1(\lambda,t)=2^{{n-2\over 2}+\alpha}  \Gamma\! \left({n\over 2}\right)\e^{\alpha t}(\e^t -1)^{-2\alpha } (\sinh t)^{\alpha-{n-1\over 2}}t^{\alpha+(n-1)/2}\sum_{j=0}^Nt^{2j}b^\alpha_j(t)\,\psi_{\alpha+(n-2)/2+j}(\lambda t)
\end{equation}
and
\begin{eqnarray*}
E^{\alpha,N}(\lambda,t) &=& 2^{{n-2\over 2}+\alpha}  \Gamma\! \left({n\over 2}\right)\e^{\alpha t}(\e^t -1)^{-2\alpha } (\sinh t)^{\alpha-(n-1)/2}\\
&&\times
\left({\mathfrak R}^\alpha(\lambda,t)+t^{\alpha+(n-1)/2}\sum_{j=0}^Nt^{2j}b^\alpha_j(t)\,\mathfrak{O}_{\alpha+(n-2)/2+j}(\lambda t)\right).
\end{eqnarray*}
Substitute  \eqref{e2.7}  back into \eqref{e2.6}, we obtain \eqref{emmm}. Estimates \eqref{e2.601} follow from \eqref{e2.701} and the estimates on the error terms ${\mathfrak R}^\alpha$ and ${\mathfrak O}_m$.

To prove part (iii), it follows from the first equation on \cite[p. 262]{Sc} that
\begin{eqnarray}\label{e2.9}
(\sinh t)^{1/2}P^{-\alpha-(n-2)/2}_{1/2+\i\lambda}(\cosh t)
&=& \bc^\alpha(\lambda)\,\e^{\i\lambda t}F\!\left(-\alpha-{n-3\over 2},\alpha+{n-1\over 2};1-\i\lambda;{-1\over \e^{2t}-1}\right)\notag\\
&&+ \bc^\alpha(-\lambda)\,\e^{-\i\lambda t}F\!\left(-\alpha-{n-3\over 2},\alpha+{n-1\over 2};1+\i\lambda;{-1\over \e^{2t}-1}\right),
\end{eqnarray}
where $F$ is the classical hypergeometric function, which has the integral representation
\begin{eqnarray*}
F(a,b;c;\nu)={\Gamma(c)\over \Gamma(b)\Gamma(c-b)}\int_0^1s^{b-1}(1-s)^{c-b-1}(1-s\nu)^{-a}\,\d s\ \ \ \ \mbox{for }\Re c>\Re b>0
\end{eqnarray*}
by \cite[p. 59, (10)]{Er}. For $\zeta\in \CC$ and $k\in{\mathbb N}$, we denote $
(\zeta)_k=\prod_{l=1}^k(\zeta+l-1)
$.
It follows from  \cite[p. 76, (10)]{Er} that   for $t>(\log 2)/2$,
\begin{eqnarray*}
F\!\left(-\alpha-{n-3\over 2},\alpha+{n-1\over 2};1+\i\lambda;{-1\over \e^{2t}-1}\right)
=\sum_{k\geq 0}{(-\alpha-(n-3)/2)_k(\alpha+(n-1)/2)_k\over (1+\i\lambda)_k\,k!}\left({-1\over \e^{2t}-1}\right)^k
\end{eqnarray*}
belongs to $S^0_\lambda$. Now we substitute  \eqref{e2.9} back into \eqref{e2.6} to  write
\begin{eqnarray*}
m^\alpha_t(\lambda)=\e^{-(n-1)t/2}\left(\e^{\i\lambda t}\bc^\alpha(\lambda)\,a^\alpha_2(\lambda,t)+\e^{-\i\lambda t}\bc^\alpha(-\lambda)\,a^\alpha_2(-\lambda,t)\right),
\end{eqnarray*}
where
\begin{eqnarray*}
a^\alpha_2(\lambda,t)&=&2^{(n-2)/2+\alpha}\Gamma\! \left({n\over 2}\right)
\e^{(\alpha+(n-1)/2)t}  (\e^t -1)^{-2\alpha} (\sinh t)^{\alpha-(n-1)/2}\\
&&\times F\!\left(-\alpha-{n-3\over 2},\alpha+{n-1\over 2};1-\i\lambda;{-1\over \e^{2t}-1}\right)
\end{eqnarray*}
belongs to $S^0_\lambda$. The proof of (iii) is concluded.
\end{proof}
\begin{rk}\label{rk:2.5}
When $\alpha=0$, $m^\alpha_t(\lambda)$ is exactly a constant multiple of the spherical function $\varphi_\lambda(t)$. Therefore, $\varphi_\lambda(t)$ has the above asymptotics in Lemma \ref{lem:2.3} with $\alpha=0$.
\end{rk}

\medskip

%%%%%%%%%%%%%%%%%%%%%%%%%%%%%%%%%%%%%%%%%%%%%%%%%%%%%%%%%%%%%%%%%%%%%%%%%%%%%%%%%%

\section{Proof of Theorem \ref{thm:1.1}}\label{sec:4}
\setcounter{equation}{0}

To prove Theorem \ref{thm:1.1}, it reduces to show  the following Lemmas~\ref{le4.1}, \ref{le4.2} and \ref{le4.3}. To do this,  we set $\rho(\cdot)=d(\cdot,\0)$ in which $\0=(1, 0, \cdots, 0)$.

Note that for $z,w\in\H^n$
\begin{eqnarray*}
[\e^tz-w]=2\e^t(\cosh t-\cosh d(z,w)).
\end{eqnarray*}
Also, for $0\leq s\leq t\leq 1$, we have
\begin{eqnarray*}
\cosh t-\cosh s={\e^t\over 2}(1-\e^{s-t})\,(1-\e^{-s-t})\sim t^2-s^2.
\end{eqnarray*}
These relations will be used often in this section.

\begin{lem}\label{le4.1}
Suppose $\alpha\in\R$ and $\fm^\alpha$ is bounded on $L^p(\H^n)$ for $1<p<\infty$.
Then we must have $\alpha>1-n+n/p$.
\end{lem}
\begin{proof}
Let $\Gamma$ be the cone vertexed at $\0$ and tangent to the ball $B_{3c_1}(\0\cdot \fa(1/2))$, where $c_1$ is a constant to be determined and $\fa(r)$ is the Lorentz boost given by
\begin{eqnarray*}
\fa(r)=\left(\begin{array}{ccc}
\cosh r & \sinh r & 0 \\
\sinh r & \cosh r & 0 \\
0 & 0 &\Id_{n-1}
\end{array}\right).
\end{eqnarray*}
Recall that $\rho(\cdot)=d(\cdot,\0)$. Define $f_\delta$ by
\begin{eqnarray*}
f_\delta(w)=
\begin{cases}
\displaystyle{(\rho(w))^{1-n-\alpha}\over -\log \rho(w)}&\mbox{in }(B_{1/2}(\0)\setminus B_\delta(\0))\cap \Gamma,\\[6pt]
0&\mbox{outside }(B_{1/2}(\0)\setminus B_\delta(\0))\cap\Gamma
\end{cases}
\end{eqnarray*}
for $\delta$ sufficiently small.

Now we estimate $M^\alpha_{\rho(z)}(f_\delta)(z)$ for $z\in B_{c_1}(\0\cdot \fa(1/2))$. To do this we let $\widetilde{S}_z\subset\Gamma$ be a cone vertexed at $\0$ and tangent to $B_{c_1}(z)$. Set $S_z=B_{1/10}(\0)\cap \widetilde{S}_z$. We choose $c_1$ sufficiently small such that for all sufficiently small $\delta>0$ and $z\in B_{c_1}(\0\cdot a(1/2))$, $S_z\cap \supp f_\delta$ is contained in the interior of $B_{\rho(z)}(z)$, and that for $w\in S_z\cap \supp f_\delta$ we have
\begin{eqnarray*}
\rho(z)-d(z,w)\sim \rho(w)>0.
\end{eqnarray*}
Thus in view of \eqref{e1.8} for such $z$ we can write
\begin{eqnarray*}
|M^\alpha_{\rho(z)}(f_\delta)(z)|
&\geq &C_\alpha\int_{S_z\cap \supp f_\delta} (\cosh\rho(z)-\cosh d(z,w))^{\alpha-1}{(\rho(w))^{1-n-\alpha}\over -\log(\rho(w))}\,\d w\\
&\geq &C_\alpha\int_\delta^{1/10} r^{\alpha-1}{r^{1-n-\alpha}\over -\log r}r^{n-1}\,\d r\\
&\geq &C_\alpha\log\log(1/\delta)
\end{eqnarray*}
using polar coordinates. Hence
\begin{eqnarray*}
\|\fm^\alpha(f_\delta)\|_p\geq \left(\int_{B_{c_1}(\0\cdot \fa(1/2))}|M^\alpha_{\rho(z)}(f_\delta)(z)|^p\,\d z\right)^{1/p}\geq C_\alpha\log\log(1/\delta).
\end{eqnarray*}
On the other hand,
\begin{eqnarray*}
\|f_\delta\|_p&\leq &C\left(\int_\delta^{1/2}{r^{(1-n-\alpha)p}\over (-\log r)^p}r^{n-1}\,\d r\right)^{1/p}\\
&\leq & \begin{cases}
C_\alpha& \mbox{for }\alpha\leq 1-n+n/p,\\[4pt]
C_\alpha\delta^{1-n-\alpha+n/p}& \mbox{for }\alpha>1-n+n/p.
\end{cases}
\end{eqnarray*}
If $\|\fm^\alpha(f)\|_p\leq C\|f\|_p$ for all $f$, then we must have $\alpha>1-n+n/p$.
\end{proof}

\begin{lem}\label{le4.2} 	Suppose $\alpha\in\R$ and $\fm^\alpha$ is bounded on $L^p(\H^n)$ for $1<p<\infty$.
Then   we must have $\alpha\geq1/p-(n-1)/2$.
\end{lem}

\begin{proof}
For $j\in {\mathbb N}$ sufficiently large, we set $R_j=\{(w_0,w_1,w'')\in\H^n:|w_1-2\cdot2^{-j}|\leq 2^{-j},|w''|\leq c_22^{-j/2}\}$, where $c_2$ is a constant to be determined. We define $g_j(w)=\indicator_{R_j}(w)$. Recall that $\rho(\cdot)=d(\cdot,\0)$.

Now we consider $z=(z_0,z_1,z'')\in\H^n$ with $\rho(z)\sim 1$ and $|z''|\leq c_22^{-j/2}$. We choose $c_2$ sufficiently small such that for all sufficiently large integer $j$ and $z$ as above, $\supp g_j$ is contained in the interior of $B_{\rho(z)}(z)$. Under this circumstance, if $w\in \supp g_j$, then we have $\rho(z)-d(z,w)>0$. Thus in view of \eqref{e1.8} for such $z$ we can write
\begin{eqnarray*}
|M^\alpha_{\rho(z)}(g_j)(z)|&\geq &C_\alpha\int_{\supp g_j}(\cosh \rho(z)-\cosh d(z,w))^{\alpha-1}\,\d w\\
&\geq &C_\alpha\int_{\rho(z)-c2^{-j}}^{\rho(z)-c'2^{-j}}2^{-(n-1)j/2}(\rho(z)-r)^{\alpha-1}\,\d r\\
&\geq &C_\alpha2^{-(\alpha+(n-1)/2)j}
\end{eqnarray*}
using polar coordinates, with $c,c'>0$ independent of $j$. Hence
\begin{eqnarray*}
\|\fm^\alpha(g_j)\|_p&\geq &\left(\int_{\{z=(z_0,z_1,z'')\in\H^n:\rho(z)\sim1,|z''|\leq c_22^{-j/2}\}}|M^\alpha_{\rho(z)}(g_j)(z)|^p\,\d z\right)^{1/p}\\
&\geq &C_\alpha2^{-(n-1)j/(2p)}2^{-(\alpha+(n-1)/2)j}.
\end{eqnarray*}
On the other hand, $\|g_j\|_p\leq C2^{-(n+1)j/(2p)}$. If $\|\fm^\alpha(f)\|_p\leq C\|f\|_p$ for all $f$, then we must have
\begin{eqnarray*}
2^{-(n-1)j/(2p)}2^{-(\alpha+(n-1)/2)j}\leq 2^{-(n+1)j/(2p)},
\end{eqnarray*}
which implies exactly $\alpha\geq1/p+(1-n)/2$.
\end{proof}

\begin{lem}\label{le4.3} 	Suppose $\alpha\in\R$ and $\fm^\alpha$ is bounded on $L^p(\H^n)$ for $1<p
<\infty$.
Then   we must have $\alpha\geq(1-n)/p$.
\end{lem}

\begin{proof}
Recall that $\rho(\cdot)=d(\cdot,\0)$. Set $E_\varepsilon=\{w\in\H^n:|\rho(w)-1+2\varepsilon|<\varepsilon\}$ for $\varepsilon$ sufficiently small. We define $h_\varepsilon(w)=\indicator_{E_\varepsilon}(w)$.

Now we consider $z\in B_{\varepsilon/100}(\0)$. Then it's obvious that $\supp h_\varepsilon$ is contained in the interior of $B_1(z)$, which indicates that if $w\in \supp h_\varepsilon$, then we have $1-d(z,w)>0$. Thus in view of \eqref{e1.8} for such $z$ we can write
\begin{eqnarray*}
|M^\alpha_1(h_\varepsilon)(z)|&\geq &C_\alpha\int_{\supp h_\varepsilon}(\cosh 1-\cosh d(z,w))^{\alpha-1}\,\d w\\
&\geq &C_\alpha\int_{1-\widetilde{c}\varepsilon}^{1-\widetilde{c}'\varepsilon}(1-r)^{\alpha-1}\,\d r\geq C_\alpha\varepsilon^\alpha
\end{eqnarray*}
using polar coordinates, with $\widetilde{c},\widetilde{c}'>0$ independent of $\varepsilon$. Hence
\begin{eqnarray*}
\|\fm^\alpha(h_\varepsilon)\|_p\geq \left(\int_{B_{\varepsilon/100}(\0)}|M^\alpha_1(h_\varepsilon)(z)|^p\,\d z\right)^{1/p}\geq C_\alpha\varepsilon^{n/p}\varepsilon^\alpha.
\end{eqnarray*}
On the other hand, $\|h_\varepsilon\|_p\leq C\varepsilon^{1/p}$. If $\|\fm^\alpha(f)\|_p\leq C\|f\|_p$ for all $f$, then we must have
\begin{eqnarray*}
\varepsilon^{n/p}\varepsilon^\alpha\leq \varepsilon^{1/p},
\end{eqnarray*}
which is exactly $\alpha\geq(1-n)/p$.
\end{proof}

\medskip

%%%%%%%%%%%%%%%%%%%%%%%%%%%%%%%%%%%%%%%%%%%%%%%%%%%%%%%%%%%%%%%%%%%%%%%%%%%%%%%%%%

\section{Proof of Theorem \ref{thm:1.2}}\label{sec:5}
\setcounter{equation}{0}

To begin with, we state some facts about the local smoothing estimates for Fourier integral operators (\cite[Chapter 8]{So1}), which will be used in the proof of Theorem \ref{thm:1.2}.
To do this, we  let  $\psi_0,\psi_1\in C_c^\infty(\R^n)$  having sufficiently small supports, and suppose     $\Phi(x,y)$ is a smooth function defined on $\supp\psi_0\times\supp\psi_1$ such that the Monge-Ampere matrix associated to  $\Phi$  is non-singular:
\begin{equation}\label{e3.1}
\left|\det \!\left(\begin{array}{cc}
0 & \partial\Phi/\partial x \\
\displaystyle{\partial\Phi\over\partial y} & \displaystyle{\partial^2\Phi\over\partial x\partial y}
\end{array}\right)\right|\geq c>0, \ \ \ \ \ x\in\supp \psi_0,\ y\in\supp\psi_1,
\end{equation}
and there is a homogeneous function $q(x, \cdot)$ with $\rank \Hess q(x, \cdot)\,\equiv n-1$ such that
\begin{equation}\label{e3.2}
q\big(x,  \Phi_x^{'} (x,y)\big)\,\equiv 1, \ \ \ \ \ x\in\supp \psi_0, \ y\in\supp\psi_1.
\end{equation}
The  property \eqref{e3.1} is usually referred to as rotational curvature while the property
\eqref{e3.2} is a simplified version of Sogge's cinematic curvature hypothesis. It follows from \cite[p. 354]{So2} that the above $\Phi(x,y)$ in \eqref{e3.2} satisfies a Carleson-Sj\"olin type hypothesis.

Given a symbol $a(\lambda,t,x,y)\in S^m_\lambda$, i.e., $a\in C^\infty(\R\times\R\times\R^n\times\R^n)$ satisfying
\begin{equation}\label{e4.2.1}
|\partial^k_\lambda \partial^l_t\partial^\gamma_{x,y} a(\lambda,t,x,y)|\leq C_{k,\gamma,l}\,(1+|\lambda|)^{m-k}\qquad\mbox{for }x\in\supp\psi_0,\ y\in\supp\psi_1\mbox{ and }t\in I
\end{equation}
for some closed interval $I\subset\R$, we can define the associated Fourier integral operator by
\begin{eqnarray*}
F_t(f)(x)= \int_{\R^n} f(y)\,K_{t}(x,y)\,\d y,
\end{eqnarray*}
where
\begin{eqnarray}\label{e3.3}
K_{t}(x,y)=\psi_0(x) \,\psi_1(y)  \int_\R \e^{\i\lambda(t-\Phi(x,y))}a(\lambda,t,x,y)\,\d\lambda.
\end{eqnarray}

The following result concerning the local smoothing estimates for Fourier integral operators is a combination of \cite[Proposition 3.2]{BeHiSo} and  \cite[Theorem 1.4]{GaLiMiXi}. Recall that  $p_n$ is given in \eqref{e1.15}, i.e.,
$p_n={2(n+1)/ (n-1)} $ for $n\geq 3$;    $p_n=4$  for $n=2$.

\begin{lem}\label{prop:3.1}
Suppose $\Phi$ satisfies \eqref{e3.1}, \eqref{e3.2}   and $a(\lambda,t,x,y)$ satisfies \eqref{e4.2.1} with $t\in I$ for some closed interval $I\subset\R$. Then for any $s<1/p_n$ with $p_n$ as in \eqref{e1.15}
\begin{eqnarray}\label{e3.4}
\|F_t(f)(x)\|_{L^{p_n}_{x,t}(\R^n\times I)}\leq C_{m,s}\|f\|_{W^{m-(n-1)/p_n-s,p_n}(\R^n)}.
\end{eqnarray}
\end{lem}

\smallskip

\begin{rk}   Beltran et al. \cite{BeHiSo} showed that for $n\geq 2$ and $2(n+1)/(n-1)\leq p<\infty$,
\begin{eqnarray}\label{e3.5}
\|F_t(f)(x)\|_{L^p_{x,t}(\R^n\times I)}\leq C_{m,s,p}\|f\|_{W^{m-(n-1)/p-s,p}(\R^n)},\ \ \ \ s<1/p
\end{eqnarray}
by establishing a variable coefficient version of the Wolff-type decoupling estimate of Bourgain-Demeter \cite{BD}. Later,   Gao et al. \cite{GaLiMiXi} improved this result for the case $n=2$, and they showed that \eqref{e3.5} holds for $4\leq p<\infty$ and $n=2$ by establishing   a variable coefficient version of the square function estimate of Guth-Wang-Zhang (see \cite{GWZ}).
%(ii)	The results in \cite{BeHiSo, GaLiMiXi} are derived using an apparently different form of the Fourier %integral operators, compared to the one we introduced here. However, in view of \cite[Proposition %6.2.4]{So1}, the two forms are equivalent modulo a $C^\infty$ error.
\end{rk}

\smallskip
Let $\beta\in C_c^\infty(\R)$ with $\supp\beta\subset [-2,-1/2]\cup [1/2,2]$ such that $\sum_{j\in \Z}\beta(2^{-j}\cdot)\,\equiv 1$ on $\R\setminus\{0\}$. We set $\beta_j=\beta(2^{-j}\cdot)$ for all integrals $j>0$ and $\beta_0=1-\sum_{j\geq 1}\beta_j$. Define
\begin{eqnarray*}
F_{t,j}(f)(x)= \int_{\R^n} f(y)\,K_{t, j}(x,y)\,\d y,
\end{eqnarray*}
where
\begin{eqnarray*}
K_{t, j}(x,y)=\psi_0(x)\, \psi_1(y)  \int_\R \beta_j(\lambda)\,\e^{\i\lambda(t-\Phi(x,y))}a(\lambda,t,x,y)\,\d\lambda.
\end{eqnarray*}

The following  Proposition \ref{prop:3.3} can be obtained by minor modifications of the proof  of \cite[Theorems 6.3.1 and 8.3.1]{So1}, so we omit it here.

\begin{prop}\label{prop:3.3}
Suppose $\Phi$ satisfies \eqref{e3.1}, \eqref{e3.2}  and $a(\lambda,t,x,y)$ satisfies \eqref{e4.2.1} with $t\in I$ for some closed interval $I\subset\R$. Then for every sufficiently small $\delta>0$ we have
\begin{eqnarray}\label{e3.666}
\left\|\sup_{t\in I}|F_{t,j}(f)|\right\|_{p_n}\leq C_{m,\delta} 2^{(m-(n-1)/p_n+\delta)j}\|f\|_{p_n}.
\end{eqnarray}
As a consequence, we have
\begin{eqnarray}\label{e3.7}
\left\|\sup_{t\in I}|F_t(f)|\right\|_{p_n}\leq C_m\|f\|_{p_n}
\end{eqnarray}
provided that   $m< (n-1)/p_n$.
\end{prop}

Now we apply Proposition~\ref{prop:3.3} to prove the following result, which implies Theorem \ref{thm:1.2} as a special case with $p=p_n$.  Precisely, we have

\begin{thm}\label{thm:5.1}
Let $p_n$ be as in \eqref{e1.15}. 	 Then $\fm^\alpha$ is bounded on $L^{p_n}(\H^n)$ whenever  $\Re\alpha>(2-n)/p_n-1/p_n^2$.
\end{thm}

To prove Theorem~\ref{thm:5.1}, it is sufficient to show the following   Proposition~\ref{prop:5.2} for the ``global'' part
of the operator $\fm^\alpha$ 
and  Proposition~\ref{prop:5.7} for the ``local" part.  Roughly speaking, our approach in the proof of Proposition~\ref{prop:5.2}   is inspired by  an argument of Ionescu in \cite{Io}, in which the crucial step is to adopt the Iwasawa coordinates to use the local smoothing of  Fourier integral operators.
We will prove Proposition~\ref{prop:5.7}  by    using the local smoothing of Fourier integral operators.

\begin{prop}\label{prop:5.2}
Let $p_n$ be as in \eqref{e1.15}. Then for every $T\geq 10$ and every small $\delta>0$,
\begin{eqnarray}\label{e5.111}
\left\|\sup_{t\in[T,T+1]}|M^\alpha_t(f)|\right\|_{p_n}\leq C_{\alpha,\delta}\e^{-(\Re\alpha+(n-2)/p_n+1/p_n^2-\delta)(n-1)p_nT}\|f\|_{p_n}.
\end{eqnarray}
As a consequence, we have
\begin{eqnarray}\label{e5.1}
\left\|\sup_{t>10}|M^\alpha_t(f)|\right\|_{p_n}\leq C_\alpha\|f\|_{p_n}
\end{eqnarray}
provided that   $\Re\alpha>(2-n)/p_n-1/p_n^2$.
\end{prop}

Let us see how to prove Proposition~\ref{prop:5.2}. 
Assume $T\geq 10$ and set
\begin{eqnarray*}
\fm^\alpha_T(f)(z)=\sup_{t\in[T,T+1]}|M^\alpha_t(f)(z)|.
\end{eqnarray*}
It follows from  (iii) of Lemma \ref{lem:2.3} that  for $t\in[T,T+1]$,
\begin{eqnarray*}
m^\alpha_t(\lambda)=\e^{-(n-1)t/2}\left(\e^{\i\lambda t}\bc^\alpha(\lambda)\,a^\alpha_2(\lambda,t)+\e^{-\i\lambda t}\bc^\alpha(-\lambda)\,a^\alpha_2(-\lambda,t)\right),
\end{eqnarray*}
where $a^\alpha_2(\lambda,t)\in S^0_\lambda$ and $\bc^\alpha$ is as in \eqref{e2.2}. Set $\phi_s\in C^\infty_c(\R)$, $0\leq \phi_s\leq 1$, $\phi_s\equiv 1$ on $[0,s+1]$ and $\phi_s\equiv 0$ outside $(-1,s+2)$.
The     convolution kernel $K^\alpha_t(r)$ of $M^\alpha_t$    can be written as the following:
\begin{eqnarray*}
K^\alpha_t(r) = \F^{-1}(m^\alpha_t)(r)
=\phi_T(r)\int_0^\infty m^\alpha_t(\lambda)\,\varphi_\lambda(r)\,|\bc(\lambda)|^{-2}\,\d \lambda 
\end{eqnarray*}
where we used the fact that   $\supp K^\alpha_t(r)\subset \{r:0\leq r\leq t\}$, see \eqref{e1.8}. Define
\begin{eqnarray}\label{ekkk}
K^\alpha_{t,j}(r) = \phi_T(r)\int_0^\infty \beta_j(\lambda)\, m^\alpha_t(\lambda)\,\varphi_\lambda(r)\,|\bc(\lambda)|^{-2}\,\d \lambda 
\end{eqnarray}
so that  $ K^\alpha_{t,j}(r) =\sum\limits_{j\geq 0} K^\alpha_{t,j}(r)$.  
By Lemma \ref{lem:2.2}
and the asymptotics of $m^\alpha_t$ and $\varphi_\lambda$ (see Lemma \ref{lem:2.3}(iii) and Remark \ref{rk:2.5}), we rewrite $  K^\alpha_{t,j}(r) = \CK^\alpha_{t,j}(r) +\CE^\alpha_{t,j}(r),
$
where
\begin{eqnarray}\label{e4.200}
\CK^\alpha_{t,j}(r)&=&
(\phi_T(r)-\phi_0(r))\,\e^{-(n-1)(r+t)/2} \notag \\
&& \times \int_\R\beta_j(\lambda)\,\bc^\alpha(\lambda)\left({\e^{\i\lambda(t+ r)}\over \bc(-\lambda)}a_2^0(\lambda,r)+{\e^{\i\lambda(t-r)}\over \bc(\lambda)}a_2^0(-\lambda,r)\right)a_2^\alpha(\lambda,t)\,\d \lambda
\end{eqnarray}
and
\begin{eqnarray}\label{e4.201}
\CE^\alpha_{t,j}(r)&=&
\phi_0(r)\,\e^{-(n-1)t/2}  \notag\\
&&\times\int_0^\infty \beta_j(\lambda)\left( \e^{\i\lambda t}\bc^\alpha(\lambda)\,a^\alpha_2(\lambda,t)+\e^{-\i\lambda t}\bc^\alpha(-\lambda)\,a^\alpha_2(-\lambda,t)\right)\varphi_\lambda(r)\,|\bc(\lambda)|^{-2}\,\d \lambda.
\end{eqnarray}
Hence,
\begin{eqnarray*}
\fm^\alpha_{T,j}(f)(z)=\sup_{t\in[T,T+1]}|f*K^\alpha_{t,j}(z)|\leq {\mathfrak A}^\alpha_{T,j}(f)(z)
+\fe^\alpha_{T,j}(f)(z),
\end{eqnarray*}
where
\begin{eqnarray*}
{\mathfrak A}^\alpha_{T,j}(f)(z)&=&\sup_{t\in[T,T+1]}|f*\CK^\alpha_{t,j}(z)|,\\
\fe^\alpha_{T,j}(f)(z)&=&\sup_{t\in[T,T+1]}|f*\CE^\alpha_{t,j}(z)|.
\end{eqnarray*}

We begin with the following lemma for the error term   $\fe^\alpha_{T,j}$.
\begin{lem}\label{lem:5.3}
For $1\leq p\leq\infty$, $0\leq \sigma<(n-1)/2$ and $j\geq 0$
\begin{eqnarray}\label{e00}
\|\fe^\alpha_{T,j}(f)\|_p\leq C_{\alpha,\sigma}2^{-(\Re\alpha+\sigma)j}\e^{-(n-1)T/2}\|f\|_p.
\end{eqnarray}
\end{lem}

\begin{proof}
By \eqref{e4.201} and Lemmas \ref{lem:2.1} and \ref{lem:2.3}(i)(iii),
\begin{eqnarray*}
\sup_{t\in[T,T+1]}|\CE^\alpha_{t,0}(r)|\leq C_\alpha\phi_0(r)\,\e^{-(n-1)T/2}.
\end{eqnarray*}
It follows from Schur's lemma  that \eqref{e00} holds  for the case $j=0$.

Now we assume $j\geq 1$. We take an even cutoff function $\gamma\in C^\infty_c(\R)$ with $\gamma\equiv 1$ on $[-1,1]$ and $\gamma\equiv0$ outside $(-2,2)$. By Lemma \ref{lem:2.3}(ii), we write
$$
\CE^\alpha_{t,j}(r)=\CE^{\alpha, 1}_{t,j}(r) +\CE^{\alpha, 2}_{t,j}(r)+ \CE^{\alpha, 3}_{t,j}(r)
$$
where
\begin{eqnarray*} 	
\CE^{\alpha, 1}_{t,j}(r) 
&=&\ \phi_0(r)\,\e^{-(n-1)t/2} \\
&&\times  \int_0^\infty \beta_j(\lambda)\,\gamma(\lambda r) \left( \e^{\i\lambda t}\bc^\alpha(\lambda)\,a^\alpha_2(\lambda,t)+\e^{-\i\lambda t}\bc^\alpha(-\lambda)\,a^\alpha_2(-\lambda,t)\right)\varphi_\lambda(r)\,|\bc(\lambda)|^{-2}\,\d \lambda,
\end{eqnarray*}
\begin{eqnarray*} 	
\CE^{\alpha, 2}_{t,j}(r) 
&=& \phi_0(r)\,\e^{-(n-1)t/2} \\ 
&&\times\int_\R\beta_j(\lambda)\,(1-\gamma(\lambda r))\left(\e^{\i\lambda (t+ r)}a^{0,N}_1(\lambda,r)+\e^{\i\lambda (t- r)}a^{0,N}_1(-\lambda,r)\right)\bc^\alpha(\lambda)\,a^\alpha_2(\lambda,t)\,|\bc(\lambda)|^{-2}\,\d \lambda,
\end{eqnarray*}

\begin{eqnarray*}	
\CE^{\alpha, 3}_{t,j}(r) 
&=&\phi_0(r)\,\e^{-(n-1)t/2}\\  
&&\times \int_0^\infty\beta_j(\lambda)\,(1-\gamma(\lambda r)) \left(\e^{\i\lambda t}\bc^\alpha(\lambda)\,a^\alpha_2(\lambda,t)+\e^{-\i\lambda t}\bc^\alpha(-\lambda)\,a^\alpha_2(-\lambda,t)\right)E^{\alpha,N}(\lambda,r)\,|\bc(\lambda)|^{-2}\,\d \lambda.
\end{eqnarray*} 	
By Lemmas \ref{lem:2.0}, \ref{lem:2.1} and \ref{lem:2.3} and integration by parts, taking $N$ sufficiently large, we obtain that $\CE^{\alpha, i}_{t,j}(r)\ (i=1,2,3)$ are all bounded by a constant multiple of $\phi_0(r)\,\e^{-(n-1)t/2}2^{-(\Re\alpha+\sigma)j}r^{-\sigma-(n+1)/2}$.
\iffalse
\begin{eqnarray*}
|\CE^{\alpha, 1}_{t,j}(r) | &\leq&  C_\alpha \phi_0(r)\,\e^{-(n-1)t/2}2^{-(\Re\alpha+\sigma)j}r^{-\sigma-(n+1)/2},\\
|\CE^{\alpha, 2}_{t,j}(r) |  &\leq&    C_{\alpha,N}\phi_0(r)\,\e^{-(n-1)t/2}\e^{-(n-1)t/2}2^{-(\Re\alpha+\sigma)j}r^{-\sigma-(n+1)/2+N},\\
|\CE^{\alpha, 3}_{t,j}(r) | &\leq&    C_{\alpha,N}\phi_0(r)\,\e^{-(n-1)t/2}\e^{-(n-1)t/2}2^{-(\Re\alpha+\sigma)j}r^{-\sigma-(n+1)/2}
\end{eqnarray*}
\fi
Hence, we obtain
\begin{eqnarray}\label{e000}
\sup_{t\in[T,T+1]}|\CE^\alpha_{t,j}(r)|\leq C_\alpha \phi_0(r)\,\e^{-(n-1)T/2}2^{-(\Re\alpha+\sigma)j}r^{-\sigma-(n+1)/2}.
\end{eqnarray}
Note that  $\Re\alpha>(1-n)/2$. By  using polar coordinates, one can derive
\begin{eqnarray*}
%&&\hspace{-1cm}
\sup_{z\in\H^n}\left\|\sup_{t\in[T,T+1]}|\CE^\alpha_{t,j}(d(z,\cdot))|\right\|_1+\sup_{w\in\H^n}\left\|\sup_{t\in[T,T+1]}|\CE^\alpha_{t,j}(d(\cdot,w))|\right\|_1
%\\&\leq &C_{\alpha}\e^{-(n-1)t/2}2^{-(\Re\alpha+\sigma) j}\int_0^2r^{-\sigma-(n+1)/2}(\sinh r)^{n-1}\,\d r
\leq C_{\alpha,\sigma}2^{-(\Re\alpha+\sigma)j}\e^{-(n-1)T/2}
\end{eqnarray*}
provided that  $\sigma<(n-1)/2$. Then we apply Schur's lemma to derive the desired result \eqref{e00}.
The proof of Lemma~\ref{lem:5.3} is concluded.
\end{proof}

Next we establish  $L^p$ boundedness of the operator $\fm^\alpha_{T,j}$ by  the following three Lemmas~\ref{lem:5.4}, \ref{lem:5.5} and \ref{lem:5.6} as three special cases $p= \infty, 2$ and $  p_n$, respectively.
Firstly, we have

\begin{lem}\label{lem:5.4}
For $T\geq 10$ and $j\geq 0$ we have
\begin{eqnarray*}
\|\fm^\alpha_{T,j}(f)\|_\infty\leq C_{\alpha}2^{-\Re\alpha j}\|f\|_\infty,
\end{eqnarray*}
\end{lem}

\begin{proof}
In view of \eqref{e4.200}, by Lemmas \ref{lem:2.1} and \ref{lem:2.3}(iii) and integration by parts, for $t\in[T,T+1]$ we have
\begin{eqnarray*}
|\CK^\alpha_{t,j}(r)|\leq C_{\alpha,N}\phi_T(r)\,\e^{-(n-1)r}2^{(1-\Re\alpha)j}(1+2^j|t-r|)^{-N}.
\end{eqnarray*}
Using polar coordinates and taking $N$ sufficiently large, for $t$ as above, we have
\begin{eqnarray*}
\sup_{w\in\H^n}\|\CK^\alpha_{t,j}(d(\cdot,w))\|_1
%\leq  C_\alpha\int_0^\infty\e^{-(n-1)r}2^{(1-\Re\alpha)j}(1+2^j|t-r|)^{-N}(\sinh r)^{n-1}\,\d r
\leq  C_{\alpha}2^{-\Re\alpha j}.
\end{eqnarray*}
By H\"{o}lder's inequality,
\begin{eqnarray*}
\|{\mathfrak A}^\alpha_{T,j}(f)\|_\infty
%=\sup_{t\in[T,T+1]}\|f*\CK^\alpha_{t,j}\|_\infty
\leq  \|f\|_\infty\sup_{w\in \H^n}\sup_{t\in[T,T+1]}\|\CK^\alpha_{t,j}(\cdot,w)\|_1\leq C_{\alpha}2^{-\Re\alpha j}\|f\|_\infty.
\end{eqnarray*}
This,  in combination with 
Lemma \ref{lem:5.3} for $\sigma=0$, % yields
concludes 
the proof of Lemma~\ref{lem:5.4}.
\end{proof}

\begin{lem}\label{lem:5.5}
For $T\geq 10$ and $j\geq 0$ we have
\begin{eqnarray*}
\|\fm^\alpha_{T,j}(f)\|_2\leq C_{\alpha}2^{-(\Re\alpha+(n-2)/2)j}\e^{-(n-1)T/2}(1+T)\,\|f\|_2.
\end{eqnarray*}
\end{lem}

\begin{proof}  By \cite[Lemma 2.4.2]{So1}, we see that  
Lemma~\ref{lem:5.5} is  a standard consequence of the following 
two estimates: 
\begin{eqnarray}\label{ekk}
\|f*K^\alpha_{t,j}\|_2\leq C_{\alpha}2^{-(\Re \alpha+(n-1)/2)j}\e^{-(n-1)T/2}(1+T)\,\|f\|_2 
\end{eqnarray}
and
\begin{eqnarray}\label{ekkkk}
\|\partial_t(f*K^\alpha_{t,j})\|_2\leq C_{\alpha}2^{-(\Re \alpha+(n-3)/2)j}\e^{-(n-1)T/2}(1+T)\,\|f\|_2 
\end{eqnarray}
for any $t\in[T,T+1]$, %and 
where 
$K^\alpha_{t,j}$ is defined in  \eqref{ekkk}.

Let us first verify \eqref{ekk}.
To do this, we set
\begin{eqnarray*}
\widetilde{K}^\alpha_{t,j}(r)=(1-\phi_T(r))\,\F^{-1}(\beta_jm^\alpha_t)(r)
\end{eqnarray*}
to be the complementary kernel of $K^\alpha_{t,j}$ such that $K^\alpha_{t,j}+\widetilde{K}^\alpha_{t,j}=\F^{-1}(\beta_j\,m^\alpha_t)$.
It follows from Lemmas~\ref{lem:2.1}, \ref{lem:2.3}(i)(iii)  and Plancherel's theorem that
\begin{equation}\label{e5.3}
\|f*(K^\alpha_{t,j}+\widetilde{K}^\alpha_{t,j})\|_2 \leq \|\beta_jm^\alpha_t\|_\infty \|f\|_2  \leq  C_\alpha 2^{-(\Re\alpha+(n-1)/2)j}(1+t)\,\e^{-(n-1)t/2} \|f\|_2.
\end{equation}
On  
the other 
hand, we apply  Lemmas~\ref{lem:2.0}, \ref{lem:2.3}(iii) to write
$$
\widetilde{K}^\alpha_{t,j}(r)=(1-\phi_T(r))\,\e^{-(n-1)(r+t)/2}\int_\R\beta_j(\lambda)\left({\e^{\i\lambda(t+ r)}\over \bc(-\lambda)}a_2^0(\lambda,r)+{\e^{\i\lambda(t-r)}\over \bc(\lambda)}a_2^0(-\lambda,r)\right)a_2^\alpha(\lambda,t)\,\d \lambda.
$$
This, together with Lemma \ref{lem:2.3}(i) and  integration by parts,   implies
\begin{eqnarray*}
\|\F(\widetilde{K}^\alpha_{t,j})\|_\infty
%&=&\left\|\int_0^\infty %\widetilde{K}^\alpha_{t,j}(r)\,\varphi_{(\cdot)}(r)\,(\sinh %r)^{n-1}\,\d r\right\|_\infty%\\
%&\leq &C_{\alpha,N}\e^{-(n-1)t/2}\int_0^\infty \e^{-(n-1)r/2}2^{-Nj}(1+|t-r|)^{-N} (1+r)\,\e^{-(n-1)r/2} (\sinh r)^{n-1}\,\d r\\
%&\leq &
\leq C_{\alpha,N}2^{-Nj}(1+t)\,\e^{-(n-1)t/2}
\end{eqnarray*}
provided $N$ sufficiently large. By   Plancherel's theorem again, we see that
\begin{equation}\label{e5.33}
\|f*\widetilde{K}^\alpha_{t,j}\|_2\leq\|\F(\widetilde{K}^\alpha_{t,j})\|_\infty\|f\|_2\leq  C_{\alpha,N}2^{-Nj}(1+t)\,\e^{-(n-1)t/2}    \|f\|_2.
\end{equation}
From these estimates \eqref{e5.3} and \eqref{e5.33},      the desired estimate  \eqref{ekk} follows readily. The proof of   \eqref{ekkkk} follows from a similar argument showing \eqref{ekk}. 
Hence the proof of Lemma~\ref{lem:5.5} is complete.
\end{proof}

The following lemma is the  essential step in proving Proposition~\ref{prop:5.2}
as a special case  $p=  p_n.$

%A crucial  observation  for the ``global'' part of $\fm^\alpha$ is the following Lemma~\ref{lem:5.6},
%whose proof  is inspired by \cite{Io}.

\begin{lem}\label{lem:5.6}
Let $p_n$ be as in \eqref{e1.15}.
For $T\geq 10$, $j\geq 0$ and sufficiently small $\delta>0$,
\begin{eqnarray}\label{e5.4}
\left\|\fm^\alpha_{T,j}(f)\right\|_{p_n}\leq C_{\alpha,\delta} 2^{-(\Re \alpha+(n-1)/p_n-\delta)j}\e^{(n-1)(1-1/p_n)T}\|f\|_{p_n}.
\end{eqnarray}
\end{lem}

\begin{proof}
By Lemma~\ref{lem:5.3} with $\sigma=(n-1)/p_n-\delta$, it suffices to show that
\begin{eqnarray}\label{e5.44}
\|{\mathfrak A}^\alpha_{T,j}(f)\|_{p_n}\leq C_{\alpha,\delta} 2^{-\varepsilon j}\e^{(n-1)(1-1/p_n)T}\|f\|_{p_n}
\end{eqnarray}
with  $\varepsilon=\Re\alpha+(n-1)/p_n-\delta$.

We now prove \eqref{e5.44}.
Roughly speaking, the favorable factor $2^{-\varepsilon j}$ in \eqref{e5.44} comes from the local smoothing estimate \eqref{e3.666}, while the unfavorable factor $\e^{(n-1)(1-1/p_n)T}$ is due to several localizations we have to make and to quantitative estimates on the   rotational curvature (the term on the RHS of \eqref{e3.1}) of defining functions of circles of radius $\approx T$.
We start by localizing the operator ${\mathfrak A}^\alpha_{T,j}$. Note that it suffices to prove that for any $\psi_0,\psi_1\in C_c^\infty(\H^n)$ with supports inside balls of radius $c_0$ ($c_0$ is sufficiently small to be determined later), one has
\begin{eqnarray}\label{e5.6}
\|\psi_0\,{\mathfrak A}^\alpha_{T,j}(\psi_1f)\|_{p_n}\leq C_{\alpha,\delta} 2^{-\varepsilon j}\e^{-(n-1)T/p_n}\|\indicator_{\supp \psi_1}f\|_{p_n}.
\end{eqnarray}

%W
To see that \eqref{e5.6} implies the desired estimate \eqref{e5.44} we fix a partition of unity $\{\eta_k\}_{k\in K}$ on $\H^n$ with each $\supp \eta_k$ contained in a ball $B_k$ of radius $c_0$ satisfying that $(1/5)B_k$ are mutually disjoint. By $cB_k$ we mean the ball with the same center as $B_k$'s and radius $c$ times as $B_k$'s. Hence it's obvious that
\begin{eqnarray}\label{e5.7}
\left\|\sum_{k\in K}\indicator_{2B_k}\right\|_\infty<\infty.
\end{eqnarray}
For any $k\in K$ we set $L_k=\{l\in K:d(\supp\eta_k,\supp\eta_l)\leq T+2\}$. Note that
\begin{eqnarray*}
\sup_{k\in K}\# L_k
%\leq {\mathrm{Vol}(B_{T+3}(z_k))\over \mathrm{Vol}((1/5)B_k)}
\leq C{\e^{(n-1)T}\over c_0^n}.
\end{eqnarray*}
%with $z_k$ being the center of $B_k$, and $\mathrm{Vol}$ is the volume induced by $\d z$.
Recall also that the Schwartz kernels $\CK^\alpha_{t,j}(d(z,w))$ vanish unless $d(z,w)\leq T+2$. Therefore
\begin{eqnarray*}
\|{\mathfrak A}^\alpha_{T,j}(f)\|^{p_n}_{p_n}&=&\int_{\H^n}\left|\sum_{k\in K}\eta_k(z)\,{\mathfrak A}^\alpha_{T,j}\!\left(\sum_{l\in L_k}\eta_lf\right)\!(z)\right|^{p_n}\d z\\
&\leq&C\,(\# L_k)^{p_n/p_n'}\sum_{k\in K}\sum_{l\in L_k}\|\eta_k{\mathfrak A}^\alpha_{T,j}(\eta_lf)\|_{p_n}^{p_n}\\
&\leq& C_{c_0,\alpha,\delta}2^{-\varepsilon p_nj}\e^{-(n-1)T}\e^{(n-1)p_nT}\|f\|_{p_n}^{p_n},
\end{eqnarray*}
which proves \eqref{e5.44}. Here we use \eqref{e5.7} in the second and the last inequalities, and the penultimate inequality follows from \eqref{e5.6}. Hence to prove \eqref{e5.44} it suffices to show \eqref{e5.6}.

Let us now prove    \eqref{e5.6}.  To do this, we define another cutoff function $\widetilde{\phi}_s\in C_c^\infty(\R)$ by $\widetilde{\phi}_s\equiv \phi_s$ on $[s,\infty)$ and $\widetilde{\phi}_s\equiv 0$ on $(-\infty,s-1]$. Note that in view of \eqref{e4.200}, we can write
\begin{equation}\label{e5.8}
\CK^\alpha_{t,j}(r)=\widetilde{\phi}_T(r)\,\e^{-(n-1)(r+t)/2}\int_\R\beta_j(\lambda)\,\e^{\i\lambda(t-r)}{\bc^\alpha(\lambda)\over \bc(\lambda)}a_2^\alpha(\lambda,t)\,a_2^0(-\lambda,r)\,\d \lambda+ \CR^\alpha_{t,j}(r),\end{equation}
where
\begin{eqnarray*}
\CR^\alpha_{t,j}(r)&=&\e^{-(n-1)(r+t)/2}\left((\phi_T(r)-\widetilde{\phi}_T(r)-\phi_0(r))\int_\R\beta_j(\lambda)\,\e^{\i\lambda(t-r)}{\bc^\alpha(\lambda)\over \bc(\lambda)}a_2^\alpha(\lambda,t)\,a_2^0(-\lambda,r)\,\d \lambda\right.\\
&&+\left.(\phi_T(r)-\phi_0(r))\int_\R\beta_j(\lambda)\,\e^{\i\lambda(t+r)}{\bc^\alpha(\lambda)\over \bc(-\lambda)}a_2^\alpha(\lambda,t)\,a_2^0(\lambda,r)\,\d \lambda\right).
\end{eqnarray*}
Lemmas \ref{lem:2.0}, \ref{lem:2.1} and \ref{lem:2.3}(iii) and integration by parts give that
\begin{eqnarray*}
\sup_{t\in[T,T+1]}|\CR^\alpha_{t,j}(r)|\leq C_{\alpha,N}\e^{-(n-1)(r+t)/2}2^{-Nj}\phi_T(r)\,(1+|t-r|)^{-N}.
\end{eqnarray*}
%It can be verified that
Hence by H\"{o}lder's inequality we have
\begin{eqnarray*}
%&&\hspace{-1cm}
\left\|\sup_{t\in[T,T+1]}|(\psi_0(\psi_1f)*\CR^\alpha_{t,j})|\right\|_{p_n}%\\
%&\leq&C_{c_0,\alpha,N}2^{-Nj}\e^{-(n-1)t/2}\|\indicator_{\supp \psi_1}f\|_{p_n}\sup_{t\in[T,T+1]}\sup_{z\in\supp\psi_0}\left(\int_{\H^n}\left[\psi_1(w)\,\phi_T(r)\,\e^{-(n-1)r/2}(1+|t-r|)^{-N}\right]^{p_n'}\,\d w\right)^{1/p_n'}\\
%&\leq &
\leq C_{c_0,\alpha,N}2^{-Nj}\e^{-(n-1)t/2}\|\indicator_{\supp \psi_1}f\|_{p_n},
\end{eqnarray*}
%where $r=d(z,w)$.
which indicates that, in view of \eqref{e5.6}, we are able to replace the convolution kernel $\CK^\alpha_{t,j}(r)$ by another kernel
\begin{equation}\label{e5.10}
\widetilde{\CK}^\alpha_{t,j}(z,w)=\psi_0(z)\,\psi_1(w)\,\e^{-(n-1)(d(z,w)+t)/2}\int_\R\beta_j(\lambda)\,\e^{\i\lambda(t-d(z,w))}b^\alpha(\lambda,t,d(z,w))\,\d \lambda
\end{equation}
with
\begin{eqnarray}\label{e4.25.1}
b^\alpha(\lambda,t,r)=\widetilde{\phi}_T(r){\bc^\alpha(\lambda)\over \bc(\lambda)}a_2^\alpha(\lambda,t)\,a_2^0(-\lambda,r)\in S^{-\Re\alpha}_\lambda
\end{eqnarray}
by Lemmas \ref{lem:2.1} and \ref{lem:2.3}(iii). To show \eqref{e5.6}, it remains to show that
\begin{equation}\label{e5.9}
\left\|\sup_{t\in[T,T+1]}\left|\int_{\H^n}f(w)\,\widetilde{\CK}^\alpha_{t,j}(\cdot,w)\,\d w\right|\right\|_{p_n}\leq C_{\alpha,\delta}2^{-\varepsilon j}\e^{-(n-1)T/p_n}\|\indicator_{\supp \psi_1}f\|_{p_n}.
\end{equation}

The estimate \eqref{e5.9} follows from the local smoothing estimate \eqref{e3.666} in Proposition \ref{prop:3.3}. To see this, we shall adopt the Iwasawa coordinates on $\H^n$ (see \cite[p. 294]{Io}). Recall that $\H^n\cong\R^{n-1}\rtimes \R$ with multiplication law given by
\begin{eqnarray*}
(v,u)\cdot(v',u')=(v+\e^u v',u+u').
\end{eqnarray*}
Also under these coordinates the geodesic distance between $(v,u)$ and $(v',u')$ is given by
\begin{equation}\label{e5.11}
d((v,u),(v',u'))=\arcosh (\e^{-u-u'}|v-v'|^2+\cosh (u-u')).
\end{equation}
Note that $d$ is invariant under left translations.

Without losing any generality, we assume the cutoff functions $\psi_0,\psi_1$ in \eqref{e5.10} satisfy $\supp\psi_0\subset B_{c_0}(0,0)$ and $\supp\psi_1\subset B_{c_0}(0,-T_0)$, where $T_0\in[T-1-2c_0,T+2+2c_0]$. Let $(v,u)\in\supp\psi_0$ and $(0,-T_0)\cdot(e^{T_0}v',u')=(v',u'-T_0)\in\supp\psi_1$. In view of \eqref{e5.11} this implies $|v|,|u|,|u'|\leq Cc_0$ and $|v'|\leq Cc_0\,\e^{-T_0}$ with $C$ independent of $T_0$. Hence by \eqref{e5.11} and the Taylor formula we can write
\begin{equation}\label{e5.12}
d((v,u),(v',u'-T_0))=T_0+u-u'+(v-v')^2 A(v,u,v',u';T_0),
\end{equation}
where $A(\cdot,\cdot,\cdot,\cdot;T_0)$ is a smooth function satisfying $A(0,0,0,0;T_0)\sim 1$ uniformly in $T_0>5$, and that all the derivatives of $A(\cdot,\cdot,\cdot,\cdot;T_0)$ at $(0,0,0,0)$ are uniformly bounded in $T_0>5$.

In view of \eqref{e5.10}, we set the function $\Phi((v,u),(v',u'))$ in \eqref{e3.3} to be $d((v,u),(v',u'-T_0))$ and the symbol $a(\lambda,t,(v,u),(v',u'))$ in \eqref{e3.3} to be $b^\alpha(\lambda,t,d((v,u),(v',u'-T_0)))$. Notice that the problem in $[(v,u),(v',u')]$-coordinates is no longer degenerate. Indeed, on the one hand, by \eqref{e5.12} the Monge-Ampere matrix of $\Phi$ turns into
\begin{eqnarray*}
\left(\begin{array}{ccccc}
0&\varepsilon_{1,2}&\dots&\varepsilon_{1,n-1}&1+\varepsilon_{1,n}\\
\varepsilon_{2,1}&-2a+\varepsilon_{2,2}&\dots&\varepsilon_{2,n-1}&\varepsilon_{2,n}\\
\vdots&\vdots&\ddots&\vdots&\vdots\\
\varepsilon_{n-1,1}&\varepsilon_{n-1,2}&\dots&-2a+\varepsilon_{n-1,n-1}&\varepsilon_{n-1,n}\\
-1+\varepsilon_{n,1}&\varepsilon_{n,2}&\dots&\varepsilon_{n,n-1}&\varepsilon_{n,n}
\end{array}\right),
\end{eqnarray*}
where $a=A(0,0,0,0;T_0)$ and each $|\varepsilon_{i,j}|\leq Cc_0$ uniformly in $T_0>5$. Hence \eqref{e3.1} is fulfilled provided $c_0$ is sufficiently small but independent of $T_0$. On the other hand, if we let $q$ be the homogeneous function in \eqref{e3.2} satisfying $q^2((v,u),\xi)=\e^{2u}(\xi_1^2+\dots+\xi_{n-1}^2)+\xi_n^2$, which is the cometric on $\H^n$, then for $u,v,u',v'$ as above one has
\begin{eqnarray*}
q\!\left((v,u),{\partial d((v,u),(v',u'-T_0))\over\partial (v,u)}\right)=|(\nabla d(\cdot,(v',u'-T_0)))(v,u)|_g\equiv 1,
\end{eqnarray*}
which is \eqref{e3.2}. Here $\nabla$ is the gradient and $|\cdot|_g$ is the norm both induced by the Riemannian metric on $\H^n$. Also, in view of \eqref{e4.25.1} and \eqref{e5.12}, $b^\alpha(\lambda,t,d((v,u),(v',u'-T_0)))$ satisfies \eqref{e4.2.1} with $m$ replaced by $-\Re\alpha$. Hence by \eqref{e3.666} in Proposition \ref{prop:3.3}, we can write
\begin{eqnarray*}
&&\int_{\R^n}\sup_{t\in[T,T+1]}\left|\int_{\R^n}f(v',u'-T_0)\,\widetilde{\CK}^\alpha_{t,j}((v,u),(v',u'-T_0))\,\d v'\d u'\right|^{p_n}\d v\d u\\
&&\qquad\leq C_{\alpha,\delta}2^{-\varepsilon p_nj}\e^{-(n-1)p_nT}\int_{\supp \psi_1}|f(v',u'-T_0)|^{p_n}\,\d v'\d u',
\end{eqnarray*}
where the term $\e^{-(n-1)p_nT}$ comes from the exponential decay in \eqref{e5.10}. Note that under the Iwasawa coordinates, the measure induced by the Riemannian metric reads $\e^{-(n-1)u}\,\d v\d u$. So
\begin{eqnarray*}
\left\|\sup_{t\in[T,T+1]}\left|\int_{\H^n}f(w)\,\widetilde{\CK}^\alpha_{t,j}(\cdot,w)\,\d w\right|\right\|_{p_n}\leq C_{\alpha,\delta}2^{-\varepsilon j}\e^{-(n-1)T/p_n}\|\indicator_{\supp \psi_1}f\|_{p_n},
\end{eqnarray*}
which is \eqref{e5.9}.
\end{proof}

Now we apply Lemma~\ref{lem:5.4}, Lemma~\ref{lem:5.5} and Lemma~\ref{lem:5.6}  to prove    Proposition~\ref{prop:5.2}.

\begin{proof}[Proof of Proposition~\ref{prop:5.2}]
Let $\delta$ be as in Lemma \ref{lem:5.6}. By Lemmas \ref{lem:5.4} and \ref{lem:5.5} and complex interpolation, we have
\begin{eqnarray}
\|\fm^\alpha_{T,j}(f)\|_{p_n}\leq C_\alpha2^{-(\Re \alpha+(n-2)/p_n-\delta)j}\e^{-(n-1)T/p_n}(1+T)^{2/p_n}\|f\|_{p_n}\label{e5.5}.
\end{eqnarray}

In order to sum over $j$, we choose  $j_0$ with $2^{j_0}=\e^{(n-1)p_n}(1+T)^{-2}$   such that the two norms in \eqref{e5.4} and \eqref{e5.5} are  equal. The result is
\begin{eqnarray*}
\|\fm^\alpha_T(f)\|_{p_n}
&\leq&\sum_{j=0}^{j_0}\|\fm^\alpha_{T,j}(f)\|_{p_n} +\sum_{j= j_0}^{\infty} \|\fm^\alpha_{T,j}(f)\|_{p_n}\\
&\leq &C_{\alpha,\delta}\e^{[-(\Re\alpha+(n-1)/p_n-\delta)(n-1)p_n+(n-1)(1-1/p_n)]T}(1+T)^{-2(\Re\alpha+(n-1)/p_n-\delta)}\|f\|_{p_n},
\end{eqnarray*}
which proves   Proposition \ref{prop:5.2}.
\end{proof}

Next we turn to handle the ``local'' part of the operator $\fm^\alpha$ by proving the following result.

\begin{prop}\label{prop:5.7} Let $p_n$ be as in \eqref{e1.15}.
If $\Re\alpha>(1-n)/p_n$, then
\begin{eqnarray}\label{e5.122}
\left\|\sup_{0<t\leq 10}|M^\alpha_t(f)|\right\|_{p_n}\leq C_\alpha\|f\|_{p_n}.
\end{eqnarray}
\end{prop}

\begin{proof}
Consider a partition of unity $\{ \varkappa_k\}_{k\in\widetilde{K}}$ on $\H^n$ with each $\supp \varkappa_k$ contained in a ball $B_k$ of radius $10$, where the balls $B_k$ satisfy that $(1/5)B_k$ are mutually disjoint. %It follows that
%\begin{eqnarray}\label{e5.13}
%\left\|\sum_{k\in\widetilde{K}}\indicator_{2B_k}\right\|_\infty<\infty.
%\end{eqnarray}
Let $\varsigma_k\in C_c^\infty(\H^n)$ such that $\varsigma_k\equiv 1$ on $2B_k$ and $\varsigma_k\equiv 0$ outside $3B_k$. By \eqref{e1.8} we have that  $M^\alpha_t( \varkappa_kf)=\varsigma_kM^\alpha_t( \varkappa_kf)$ for  $0<t\leq 10$.

To prove Proposition~\ref{prop:5.7},   it is enough to show that for all $\alpha\in\CC$ with $\Re\alpha>(1-n)/p_n$ there holds
% \begin{eqnarray}\label{e5.14}
% \left\|\sup_{0<t\leq 10}|\varsigma_kM^\alpha_t( \varkappa_kf)|\right\|_{p_n}\leq C_\alpha\|\indicator_{\supp\varkappa_k}f\|_{p_n} 
% \end{eqnarray}
% which reduces to show  
\begin{eqnarray}\label{e5.15}
\left\|\sup_{t\in[1,10]}|\psi_0\,M^\alpha_t( \psi_1f)|\right\|_{p_n}
+ \left\|\sup_{0<t\leq 1}|\psi_0\,M^\alpha_t( \psi_1f)|\right\|_{p_n} \leq C_\alpha\|\indicator_{\supp\psi_1}f\|_{p_n}
\end{eqnarray}
where $\psi_0,\psi_1\in C^\infty_c(\H^n)$ have sufficiently small supports.

\medskip

Let us prove \eqref{e5.15} for the case $1\leq t\leq 10$. Since the convolution kernel $K^\alpha_t$ of $M^\alpha_t$ satisfies $\supp K^\alpha_t(r)\subset \{r:0\leq r\leq t\}$ (see \eqref{e1.7}), we  write
\begin{eqnarray*}
K^\alpha_t(r) =\phi_t(r)\int_0^\infty m^\alpha_t(\lambda)\,\varphi_\lambda(r)\,|\bc(\lambda)|^{-2}\,\d\lambda=:\mathscr{B}^\alpha_t(r)+\mathscr{D}^\alpha_t(r),
\end{eqnarray*}
where by Lemma \ref{lem:2.0} and the asymptotics of $m^\alpha_t$ and $\varphi_\lambda$ (see Lemma \ref{lem:2.3}(iii) and Remark \ref{rk:2.5})
\begin{eqnarray}\label{e5.160}
\mathscr{B}^\alpha_t(r)=(\phi_t(r)-\phi_0(r))\,\e^{-(n-1)(r+t)/2}\int_\R\e^{\i\lambda(t-r)}{\bc^\alpha(\lambda)\over\bc(\lambda)}a^\alpha_2(\lambda,t)\,a^0_2(-\lambda,r)\,\d\lambda,
\end{eqnarray}
$\mathscr{D}^\alpha_t(r)=K^\alpha_t(r)-\mathscr{B}^\alpha_t(r)$ and $\phi_s$ is the cutoff function as above.

Observe  that in terms of \eqref{e5.15}, $\mathscr{D}^\alpha_t$
is an error term. To see this, by Lemmas \ref{lem:2.0}, \ref{lem:2.1} and \ref{lem:2.3} 
together with 
an argument in the proof of Lemma \ref{lem:5.3},  we have  
\begin{eqnarray}\label{e5.162}
|\mathscr{D}^\alpha_t(r)|\leq C_\alpha(\phi_{10}(r)+\phi_0(r)\,r^{\Re\alpha-(n+1)/2}),
\end{eqnarray}
which implies  
\begin{eqnarray*}
\sup_{z\in\H^n}\left\|\sup_{t\in[1,10]}|\mathscr{D}^\alpha_t(d(z,\cdot))|\right\|_1+\sup_{w\in\H^n}\left\|\sup_{t\in[1,10]}|\mathscr{D}^\alpha_t(d(\cdot,w))|\right\|_1\leq C_\alpha.
\end{eqnarray*}
By Schur's lemma,   we obtain 
\begin{eqnarray}\label{e5.161}
\left\|\sup_{t\in[1,10]}|\psi_0\,((\psi_1f)* \mathscr{D}^\alpha_t)|\right\|_{p_n}\leq C_\alpha\|f\|_{p_n}.
\end{eqnarray}

Let us turn to the term $\mathscr{B}^\alpha_t$.  In view of 
the remark following \cite[Corollary 2.2]{So2}, under local coordinates, for $1\leq t\leq 10$, if we set the function $\Phi$ in \eqref{e3.3} to be $d(z,w)$ and the symbol $a$ in \eqref{e3.3} to be
$$
(\phi_t(d(z,w))-\phi_0(d(z,w))){\bc^\alpha(\lambda)\over\bc(\lambda)}a^\alpha_2(\lambda,t)\,a^0_2(-\lambda,d(z,w)),
$$
then the conditions on $\Phi$ and $a$ in Proposition \ref{prop:3.3} are fulfilled. Therefore we can invoke \eqref{e3.7} in Proposition \ref{prop:3.3} to obtain
\begin{eqnarray*}
\left\|\sup_{t\in[1,10]}|\psi_0\,((\psi_1f)*\mathscr{B}^\alpha_t)|\right\|_{p_n}\leq C_\alpha\|\indicator_{\supp \psi_1}f\|_{p_n}\ \ \ \ \mbox{for }\Re\alpha>{1-n\over p_n}.
\end{eqnarray*}
This, combined with \eqref{e5.161}, concludes the proof of \eqref{e5.15} for the case $1\leq t\leq 10$.

\medskip

Now we consider \eqref{e5.15} for the case $0<t\leq 1$.   Let $\gamma$ be the cutoff function in the proof of Lemma \ref{lem:5.3}. We define another cutoff function $\widetilde{\gamma}\in C_c^\infty(\R_+)$ with $\widetilde{\gamma}\equiv 1$ on $[1/2,\infty)$ and $\widetilde{\gamma}\equiv 0$ on $[0,1/4]$. Let $K^\alpha_t(r)$ be the convolution kernel of $M^\alpha_t$. Note that in view of \eqref{e1.8}, $\supp K^\alpha_t(r)\subset \{r:0\leq r\leq t\}$. So we can write
\begin{eqnarray*}
K^\alpha_t(r)=\gamma\!\left({r\over t}\right)\int_0^\infty m^\alpha_t(\lambda)\,\varphi_\lambda(r)\,|\bc(\lambda)|^{-2}\,\d \lambda=\mathscr{P}^\alpha_t(r)+\mathscr{W}^\alpha_t(r),
\end{eqnarray*}
where by the asymptotics of $m^\alpha_t$ and $\varphi_\lambda$ (see Lemma \ref{lem:2.3}(ii) and Remark \ref{rk:2.5})
\begin{eqnarray*}
\mathscr{P}^\alpha_t(r)=\widetilde{\gamma}\!\left({r\over t}\right)\int_\R\e^{\i\lambda(t-r)}a^{\alpha,N}_1(\lambda,t)\,a^{0,N}_1(-\lambda,r)\,(1-\gamma(\lambda r))\,|\bc(\lambda)|^{-2}\,\d\lambda
\end{eqnarray*}
and $\mathscr{W}^\alpha_t(r)=K^\alpha_t(r)-\mathscr{P}^\alpha_t(r)$.

Observe that in terms of \eqref{e5.15}, $\mathscr{W}^\alpha_t$ is an error term. To see this, by Lemmas \ref{lem:2.0}, \ref{lem:2.1} and \ref{lem:2.3}, an argument in  the proof of Lemma \ref{lem:5.3} shows that
\begin{eqnarray*}
|\mathscr{W}^\alpha_t(r)|\leq C_\alpha t^{-n}\indicator_{[0,2t]}(r)\left(1+\left({t\over r}\right)^{-\Re\alpha+(n+1)/2}\right),
\end{eqnarray*}
and so
\begin{eqnarray}\label{e5.166}
\left\|\sup_{0<t\leq 1}|\psi_0\,((\psi_1f)*\mathscr{W}^\alpha_t)|\right\|_{p_n}\leq C_\alpha \|f\|_{p_n}\ \ \ \mbox{for }\Re\alpha>{1-n\over p_n}.
\end{eqnarray}

Now we claim that for $ \Re\alpha>{(1-n)/p_n}$,
\begin{eqnarray}\label{e5.17}
\left\|\sup_{0<t\leq 1}|\psi_0\,((\psi_1f)*\mathscr{P}^\alpha_t)|\right\|_{p_n}\leq C_\alpha \|f\|_{p_n}.
\end{eqnarray}
To prove \eqref{e5.17}, we observe that 
\begin{eqnarray*}
\mathscr{P}^\alpha_t(r)=\widetilde{\widetilde{\gamma}}\!\left({r\over t}\right)t^{-n}\int_{\R}\e^{\i\lambda(1-r/t)}|\lambda|^{-\Re\alpha}\widetilde{a}^\alpha_1(\lambda,t)\,\widetilde{a}^0_1\!\left(\lambda {r\over t},r\right)\left(1-\gamma\!\left(\lambda {r\over t}\right)\right)\fc\!\left({\lambda\over t}\right)\d\lambda,
\end{eqnarray*}
where $\widetilde{\widetilde{\gamma}}(s)=s\,\widetilde{\gamma}(s)$,
$\fc(\lambda)=|\lambda|^{-(n-1)}|\bc(\lambda)|^{-2}$ and $\widetilde{a}^\alpha_1(\lambda,t)=|\lambda|^{-\Re\alpha-(n-1)/2}a^{\alpha,N}_1(\lambda/t, t)$.  By Lemma \ref{lem:2.1} and \eqref{e2.702}, we see that $\fc(\lambda),\widetilde{a}^\alpha_1(\lambda,t)\in S^0_\lambda$.
Without losing any generality, we assume the supports of $\psi_0$ and $\psi_1$ in \eqref{e5.17} are in some neighborhoods of $\0$. Here we adopt the hyperboloid coordinates. Let
\begin{eqnarray*}
\begin{array}{cccc}
\tau:&\R^n&\to&\H^n,\\
&x&\mapsto&(\sqrt{1+|x|^2},x).
\end{array}
\end{eqnarray*}
Under this coordinates, the geodesic distance reads
\begin{eqnarray*}
d(\tau(x),\tau(y))=\arcosh(\sqrt{1+|x|^2}\sqrt{1+|y|^2}-x\cdot y).
\end{eqnarray*}	
We set $\widetilde{d}^t(x,y)=t^{-1}d(t\,\tau(x),t\,\tau(y))$ and
\begin{eqnarray*}
\widetilde{\mathscr{P}}^\alpha_t(r)=t^{n}\mathscr{P}^\alpha_t(tr)=\widetilde{\widetilde{\gamma}}(r)\int_{\R}\e^{\i\lambda(1-r)}|\lambda|^{-\Re\alpha}\widetilde{a}^\alpha_1(\lambda,t)\,\widetilde{a}^0_1(\lambda r,tr)\,(1-\gamma(\lambda r))\,\fc\!\left({\lambda\over t}\right)\d\lambda.
\end{eqnarray*}
Then \eqref{e5.17} is equivalent to
\begin{equation}\label{e5.18}
\left\|\sup_{0<t\leq 1}\left|\widetilde{\psi}_0(\cdot)\int_{\R^n}\widetilde{\psi}_1(y)\,g(y)\,\widetilde{\mathscr{P}}^\alpha_t(\widetilde{d}^t(\cdot,y))\,\d y\right|\right\|_{L^{p_n}(\R^n)}\leq C_{\alpha}\|g\|_{L^{p_n}(\R^n)}
\end{equation}
for $ \Re\alpha>{(1-n)/p_n}.$

To prove \eqref{e5.18},  we set
\begin{eqnarray*}
b(\lambda,t,x,y)=\widetilde{\widetilde{\gamma}}(\widetilde{d}^t(x,y))\,(1-\gamma(\lambda \widetilde{d}^t(x,y)))\,|\lambda|^{-\Re\alpha}\widetilde{a}^\alpha_1(\lambda,t)\,\widetilde{a}^0_1(\lambda \widetilde{d}^t(x,y),t\widetilde{d}^t(x,y))\,\fc\!\left({\lambda\over t}\right).
\end{eqnarray*}
In view of \cite[Proposition 6.2.4]{So1}, modulo a $C^\infty$ error, for $0<t\leq1$, $\mathscr{P}^\alpha_t(\widetilde{d}^t(x,y))$ can be written as the sum of %finitely
finite
terms, each of which is of the form
\begin{eqnarray}\label{e5.21}
\widetilde{\widetilde{\gamma}}(\widetilde{d}^t(x,y))\,\widetilde{\widetilde{\psi}}_0(x)\,\widetilde{\widetilde{\psi}}_1(y)\int_{\R^n}\e^{\i(\widetilde{\varphi}(x,t,\xi)-y\cdot\xi)}\widetilde{b}(x,t,\xi)\,\d\xi,
\end{eqnarray}
where   
\begin{eqnarray*}
|\partial^{\gamma_1}_x\partial^{\gamma_2}_\xi\widetilde{b}(x,t,\xi)|\leq C_{\alpha,\gamma_1,\gamma_2}(1+|\xi|)^{-\Re\alpha-(n-1)/2-|\gamma_2|}\qquad\mbox{for }0<t\leq 2^{-j_0},
\end{eqnarray*}
\begin{eqnarray*}
|\partial^{\gamma_1}_x\partial^k_t\partial^{\gamma_2}_\xi\widetilde{b}(x,t,\xi)|\leq C_{\alpha,k,\gamma_1,\gamma_2}(1+|\xi|)^{-\Re\alpha-(n-1)/2-|\gamma_2|}\qquad\mbox{for }t\in [2^{-j_0},1] 
\end{eqnarray*}
for an integer $j_0$ 
to be
chosen later, and  $\widetilde{\varphi}$ is homogeneous of degree $1$ on $\xi$ satisfying $\det({\partial^2\widetilde{\varphi}/\partial x\partial\xi})\ne 0$. If we set $T_t\ (0<t\leq 1)$ to be the operator given by integration against the Schwartz kernel \eqref{e5.21}, then by \cite[Proposition 3.2]{BeHiSo} and \cite[Theorem 1.4]{GaLiMiXi},
\begin{eqnarray*}
\|T_t(g)(x)\|_{L^{p_n}_{x,t}(\R^n\times[2^{-j_0},1])}\leq C_\alpha\|g\|_{W^{-\Re\alpha-(n-1)/p_n-s,p_n}(\R^n)}\qquad\mbox{for any }s<{1\over p_n}.
\end{eqnarray*}
We then use an argument in 
the proof of 
\cite[Theorem 6.3.1]{So1} to obtain
$
\|\sup_{t\in[2^{-j_0},1]}|T_t(g)|\|_{p_n}\leq C_\alpha \|g\|_{p_n}.
$
Hence,  the proof of  \eqref{e5.18} reduces to 
showing 
\begin{eqnarray}\label{e5.23}
\left\|\sup_{0<t\leq 2^{-j_0}}|T_t(g)|\right\|_{p_n}\leq C_\alpha \|g\|_{p_n}.
\end{eqnarray}
For this,  we need to consider $\mathscr{P}^\alpha_t(\widetilde{d}^t(x,y))$ when $t\to0$. Note that $\widetilde{d}^t(x,y)\to |x-y|$ as $t\to 0$ uniformly in $x,y$ belonging to a fixed compact set. We also have  that $\fc(\lambda)\to |C_n'|^{-2}$ as $\lambda\to\infty$, with $C_n'=2^{n-2}\Gamma(n/2)/\sqrt\uppi$ being the constant in \eqref{e2.1}.  It tells us  that $\widetilde{K}^\alpha_0(\widetilde{d}^0(x,y))$ is of the form
\begin{eqnarray*}
&&\widetilde{\widetilde{\gamma}}(|x-y|)\int_{\R}\e^{\i\lambda(1-|x-y|)}|\lambda|^{-\Re\alpha}\widetilde{a}^\alpha_1(\lambda,0)\,\widetilde{a}^0_1(\lambda |x-y|,0)\,(1-\gamma(\lambda |x-y|))\,\d\lambda.
\end{eqnarray*}
This, together with \cite[Proposition 6.2.4]{So1} and the asymptotics of the   Fourier transform of the spherical measure on $\R^n$ derived using \cite[Chapter VIII, (15) and (25)]{St1}, indicates that $\widetilde{\varphi}$ in \eqref{e5.21} satisfies $\widetilde{\varphi}(x,0,\xi)=x\cdot\xi\pm|\xi|$. Choosing $j_0$ sufficiently large and noting that $\widetilde{\varphi}(x,t,\cdot)$ is homogeneous of degree $1$, we use  Taylor's formula to derive that the phase function in \eqref{e5.21} satisfies
\begin{eqnarray*}
\widetilde{\varphi}(x,t,\xi)=x\cdot\xi\pm|\xi|+O(t|\xi|) \qquad\mbox{for }0<t\leq 2^{-j_0}.
\end{eqnarray*}
Then  we rescale the integral representation of $T_t$ into the form of the operator $\mathcal{F}_t$ on \cite[p. 192]{So1}. Using the argument showing \cite[Corollary 6.3.3]{So1}, together with \cite[Proposition 3.2]{BeHiSo} and \cite[Theorem 1.4]{GaLiMiXi}, we obtain \eqref{e5.23}, and the estimate \eqref{e5.18} is valid. This proves \eqref{e5.17} and finishes the proof of \eqref{e5.15} for the case $0<t\leq 1$. Hence, the proof of Proposition~\ref{prop:5.7} is concluded. 
\end{proof}

\begin{proof}[Proof of Theorem  \ref{thm:5.1}]
This is a direct consequence of \eqref{e5.1} and Proposition~\ref{prop:5.7}.
\end{proof}

We finally %present the endgame in the
come to the
\begin{proof}[Proof of Theorem~\ref{thm:1.2}]
We shall use the interpolation argument in \cite{Ko}. In view of \cite[Theorem 3]{Ko}, $\fm^\alpha$ is bounded on $L^p(\H^n)$ for $\Re\alpha>0$ and $p=\infty$, or $\Re\alpha>(2-n)/2$ and $p=2$. Then we interpolate between these results and Theorem \ref{thm:5.1} to derive Theorem \ref{thm:1.2}.
\end{proof}

\iffalse
\begin{rk}
There is a gap between the borderline of the range of $(\Re\alpha,p)$ in Theorem \ref{thm:1.2} and the range $(\alpha,p)$ in Theorem \ref{thm:1.1} 2). The former one doesn't seem like it's sharp. The main reason is that the localization argument in the proof of Lemma \ref{lem:5.6} brings an ``unfavorable'' exponential growth in the estimate \eqref{e5.4}, as pointed out in the beginning of the proof of Lemma \ref{lem:5.6}. In addition, when $n\geq 3$, the local smoothing estimates for general Fourier integral operators is not sharp enough in the following sense. As is pointed out in \cite{GaLiMiXi}, the estimate \eqref{e3.4} with $p$ replaced by $p_n$ fails for all $s<1/p$, if
\begin{eqnarray*}
p<\begin{cases}
\displaystyle{2(n+1)\over n-1}&\mbox{if }n\mbox{ is odd},\\
\displaystyle{2(n+2)\over n}&\mbox{if }n\mbox{ is even}.
\end{cases}
\end{eqnarray*}
However, to achieve the range of $(\alpha,p)$ in Theorem \ref{thm:1.1} 2), by the above argument at least the estimate should hold for all $s<1/p$ in the larger range $2n/(n-1)\leq p<\infty$. Note that this range of $p$ is exactly the one in the local smoothing conjecture for the wave equation on $\R^n$ proposed in \cite{So1}.
\end{rk}
\fi

\bigskip

\noindent
{\bf Acknowledgements}: The authors    were supported  by National Key R$\&$D Program of China 2022YFA1005700. P. Chen was supported by NNSF of China 12171489. M. Shen was supported by China Postdoctoral Science Foundation 2024M761509. L. Yan was supported by NNSF
of China 12571111. The authors  thank X. Chen, N. Liu, L. Roncal, L. Song and H.-W. Zhang 
  for helpful  discussions and suggestions.

\bigskip

\bibliographystyle{plain}

\end{document}